\documentclass{amsart}

\newcommand{\cJ}{\mathcal{J}}

\newcommand{\bL}{\mathbb{L}}

\newcommand{\bQ}{\mathbb{Q}}\newcommand{\bR}{\mathbb{R}}
\newcommand{\bS}{\mathbb{S}}

\newcommand{\bZ}{\mathbb{Z}}

\usepackage[dvipsnames,svgnames,x11names,hyperref]{xcolor}
\usepackage{lmodern,mathtools,url,graphicx,verbatim,amssymb,enumerate,stmaryrd,nicefrac}
\usepackage[pagebackref,colorlinks,citecolor=Mahogany,linkcolor=Mahogany,urlcolor=Mahogany,filecolor=Mahogany]{hyperref}
\usepackage{microtype}
\usepackage[all,cmtip]{xy}
\usepackage[margin=1.5in]{geometry}
\usepackage{setspace}
\usepackage{caption}
\usepackage{tikz,tikz-cd}
\usetikzlibrary{matrix,calc,positioning,arrows,decorations.pathreplacing,fit,patterns}

\newtheorem{theorem}{Theorem}[section]
\newtheorem{lemma}[theorem]{Lemma}
\newtheorem{proposition}[theorem]{Proposition}
\newtheorem{corollary}[theorem]{Corollary}
\newtheorem*{corollary*}{Corollary}

\newtheorem{atheorem}{Theorem}

\newtheorem{acorollary}[atheorem]{Corollary}

\theoremstyle{definition}
\newtheorem{definition}[theorem]{Definition}
\newtheorem{example}[theorem]{Example}
\newtheorem{remark}[theorem]{Remark}
\newtheorem{notation}[theorem]{Notation}

\newcommand{\half}{\nicefrac{1}{2}}
\newcommand{\cat}[1]{\mathsf{#1}}
\newcommand{\mr}[1]{{\rm #1}}

\newcommand{\fS}{\mathfrak{S}}

\newcommand{\diff}{\mr{Diff}}
\newcommand{\pl}{\mr{PL}}
\newcommand{\topo}{\mr{Top}}

\newcommand{\longto}{\longrightarrow}

\title{Some finiteness results for groups of automorphisms of manifolds}

\author{Alexander Kupers}
\thanks{Alexander Kupers was partially supported by a William R. Hewlett Stanford Graduate Fellowship, Department of Mathematics, Stanford University, by NSF grant DMS-1105058, by the Danish National Research Foundation through the Centre for Symmetry and Deformation (DNRF92), by the European Research Council (ERC) under the European Union's Horizon 2020 research and innovation programme (grant agreement No.\ 682922), and by NSF grant DMS-1803766.}
\email{kupers@math.harvard.edu}
\address{Harvard University \\
	Department of Mathematics \\
	One Oxford Street \\
	Cambridge MA, 02138 \\USA}

\date{\today}

\begin{document}

\begin{abstract}We prove that in dimension $\neq 4,5,7$ the homology and homotopy groups of the classifying space of the topological group of diffeomorphisms of a disk fixing the boundary are finitely generated in each degree. The proof uses homological stability, embedding calculus and the arithmeticity of mapping class groups. From this we deduce similar results for the homeomorphisms of $\bR^{n}$ and various types of automorphisms of $2$-connected manifolds. \end{abstract}

\maketitle

\tableofcontents

\section{Introduction} Inspired by work of Weiss on Pontryagin classes of topological manifolds \cite{weissdalian}, we use several recent advances in the study of high-dimensional manifolds to prove a structural result about diffeomorphism groups. We prove the classifying spaces of such groups are often ``small'' in one of the following two algebro-topological senses:

\begin{definition}Let $X$ be a path-connected space.\begin{itemize}
\item $X$ is said to be of \emph{homologically finite type} if for all $\bZ[\pi_1(X)]$-modules $M$ that are finitely generated as abelian groups, $H_*(X;M)$ is finitely generated in each degree.
\item $X$ is said to be of \emph{finite type} if $\pi_1(X)$ is finite and $\pi_i(X)$ is finitely generated for $i \geq 2$.
\end{itemize}
Being of finite type implies being of homologically finite type, see Lemma \ref{lem.pfinhfin}.
\end{definition}

Let $\diff(-)$ denote the topological group of diffeomorphisms in the $C^\infty$-topology, $\pl(-)$ the simplicial group of PL-homeomorphisms, $\topo(-)$ the topological group of homeomorphisms in the compact-open topology. A subscript $\partial$ means we restrict to the subgroup of those automorphisms that are the identity on the boundary, and a superscript $+$ means we restrict to orientation-preserving automorphisms. The following solves Problems 1(b) and 1(d) of Burghelea in \cite{burgheleaproblems}.

\begin{atheorem}\label{thm.main} Let $n \neq 4,5,7$, then $B\diff_\partial(D^{n})$ is of finite type.\end{atheorem}

\begin{acorollary}Let $n \neq 4,5,7$, then $B\diff(S^{n})$ is of finite type\end{acorollary}

\begin{acorollary}\label{cor.diff} Let $n \neq 4,5,7$. Suppose that $M$ is a closed 2-connected oriented smooth manifold of dimension $n$, then $B\mr{Diff}^+(M)$ is of homologically finite type.\end{acorollary}

It is convention to denote $\pl(\bR^n)$ by $\pl(n)$ and $\topo(\bR^n)$ by $\topo(n)$.

\begin{acorollary}Let $n \neq 4,5,7$, then $B\mr{Top}(n)$ and $B\mr{PL}(n)$ are of finite type.\end{acorollary}

\begin{acorollary}Let $n \neq 4,5,7$, then $B\mr{Top}(S^n)$ and $B\mr{PL}(S^n)$ are of finite type.\end{acorollary}

\begin{acorollary}\label{cor.pltop} Let $n \neq 4,5,7$. Suppose that $M$ is a closed 2-connected oriented smoothable manifold of dimension $n$, then $B\pl^+(M)$ and $B\topo^+(M)$ are of homologically finite type.\end{acorollary}

For completeness, Propositions \ref{prop.haut} and \ref{prop.block} give similar results for homotopy automorphisms, block automorphisms and the quotient of block automorphisms by automorphisms. Furthermore, using the relationship between diffeomorphisms and algebraic K-theory, we deduce finiteness properties of the spectra $\mr{Wh}^\mr{Diff}(*)$ and $A(*)$. 

%The following result was previously proved by Dwyer \cite{dwyertwisted}.
%
%%\begin{acorollary}We have that $K_i(\bZ) \otimes \bQ$ is finite-dimensional for all $i$.\end{acorollary}
%
%\begin{acorollary}We have that $\mr{Wh}^\mr{Diff}(*)$ and $A(*)$ are of finite type.\end{acorollary}

The input to our proofs is a number of deep theorems in manifold theory: the homological stability results of Galatius and Randal-Williams \cite{grwstab1,grwcob} and Botvinnik and Perlmutter \cite{perlmutterbotvinnik,perlmutterstab}, the embedding calculus of Weiss \cite{weissembeddings,weisspedrosheaves}, the excision estimates of Goodwillie and Klein \cite{goodwillieklein}, and the arithmeticity results of Sullivan \cite{sullivaninf,trianta}. 

These are combined with a fiber sequence we call the \emph{Weiss fiber sequence}, as it underlies the work of Weiss in \cite{weissdalian}. This fiber sequence expresses the diffeomorphisms of a disk as the difference between the diffeomorphisms of a manifold and certain self-embeddings of that manifold. The setup is as follows: let $M$ be an $n$-dimensional smooth manifold with non-empty boundary $\partial M$ and an embedded disk $D^{n-1} \subset \partial M$, then $\mr{Emb}^{\cong}_{\half\partial}(M)$ is  the space of embeddings $M \hookrightarrow M$ that are the identity on $\partial M \backslash \mr{int}(D^{n-1})$ and isotopic fixing $\partial M \backslash \mr{int}(D^{n-1})$ to a diffeomorphism that fixes the boundary. In Section \ref{sec.weissfiber} we construct a fiber sequence
\begin{equation} \label{eqn.weissfib} B\mr{Diff}_\partial(D^{n}) \longto B\mr{Diff}_\partial (M) \longto B\mr{Emb}^{\cong}_{\half\partial}(M)\end{equation}
and show it deloops once to a fiber sequence
\begin{equation} \label{eqn.weissfib2} B\mr{Diff}_\partial (M) \longto B\mr{Emb}^{\cong}_{\half\partial}(M) \longto B(B\mr{Diff}_\partial(D^{n}),\natural)\end{equation}
with $\natural$ denoting boundary connected sum. The advantage of (\ref{eqn.weissfib2}) is that the base is 1-connected, so we can use technical input about spaces of homologically and homotopically finite type discussed in Section \ref{sec.fintype}. 

The aforementioned deep theorems tells us there are manifolds $W_{g,1}$ and $H_g$ such that the group of the diffeomorphisms of a disk is the only unknown term in (\ref{eqn.weissfib2}). Thus we can solve for $B\mr{Diff}_{\partial}(D^n)$. This insight is due to Weiss, and was used by him in \cite{weissdalian} to study the rational cohomology of $B\mr{Top}(n)$, related to $\mr{Diff}_\partial(D^{n})$ by smoothing theory. We instead use it to study finiteness properties of $\mr{Diff}_\partial(D^{n})$.

\subsection{Historical remarks} We discuss related results in the literature. This discussion is not complete and does not cover the results used in our argument, e.g.\ \cite{grwcob,grwstab1,weissdalian,perlmutterbotvinnik,perlmutterstab}. We start with diffeomorphisms of a disk. As often in smooth manifold theory, there is the following trichotomy: (i) \emph{low dimensions} $\leq 3$, (ii) \emph{dimension 4}, (iii) \emph{high dimensions} $n \geq 5$.

\begin{enumerate}[(i)]
\item In low dimensions $\leq 3$, the diffeomorphisms of a disk are contractible. For $n=1$ this is folklore, for $n=2$ this is due to Smale \cite{smaledisk}, and for $n=3$ to Hatcher \cite{hatcherdisk}.
\item In dimension 4, nothing is known.
\item In high dimensions $\geq 5$, the homotopy groups of diffeomorphisms of a disk are only understood in low degrees. 

The connected components are known: the group $\pi_0( \mr{Diff}_\partial(D^n))$ is isomorphic to the group $\Theta_{n+1}$ of diffeomorphism classes of homotopy $(n+1)$-spheres under connected sum \cite{levine}. The $h$-cobordism theorem proves surjectivity and Cerf's theorem proves injectivity \cite{cerf}. This group is known to be finite abelian and is related to the stable homotopy groups of spheres \cite{kervairemilnor}.

The higher homotopy groups can be determined in a range using the pseudoisotopy stability theorem \cite{igusastab} and algebraic $K$-theory \cite{waldhausenalgk}. In \cite{farrellhsiang}, Farrell and Hsiang  proved that in the so-called concordance stable range $0 < i < \frac{n}{6}-7$ (improved to $0<i<\min(\frac{n-7}{2},\frac{n-4}{3})$ by Igusa, and recently Randal-Williams gave an upper bound \cite{oscarupper})
\[ \pi_{i}(\mr{Diff}_\partial(D^n)) \otimes \bQ \cong \begin{cases} 0 & \text{if $n$ is even,} \\
K_{i+2}(\bZ) \otimes \bQ & \text{if $n$ is odd.} \end{cases}\]
The latter is given by $\bQ$ if $i \equiv 3 \pmod 4$ and $0$ otherwise \cite{borelstable}. Igusa's work on higher torsion invariants gives an alternative proof \cite[Section 6.5]{igusatorsion}.

There are known examples of non-zero homotopy groups. Rationally, Watanabe showed that  $\pi_{2n-2}(B\mr{Diff}_\partial(D^{2n+1})) \otimes \bQ \neq 0$ for many $n$ \cite{watanabe1}, giving lower bounds on the dimension of $\pi_{kn-k}(B\mr{Diff}_\partial(D^{2n+1})) \otimes \bQ$ for any $k \geq 2$ \cite{watanabe2}. %One is tempted to conjecture a relationship between these classes and the embedding calculus employed in this paper.
Novikov found torsion elements in $\pi_i(B\mr{Diff}_\partial(D^n))$ \cite{novikovspheres}. Burghelea and Lashof proved that there is an infinite sequence $(p_i,k_i,n_i)$ with $p_i$ an odd prime and $\lim_{i \to \infty} n_i/k_i = 0$ such that $\pi_{k_i}(\mr{Diff}_\partial(D^{n_i})) \otimes \bZ/p_i \bZ \neq 0$ \cite[Section 7]{burglashof}, improving on an earlier result by Miller \cite{millerdiff}. Crowley, Schick and Steimle proved that for $n \geq 6$, there are infinitely many $i$ such that $\pi_i( \mr{Diff}_\partial(D^n))$ is non-zero, containing an element of order 2 \cite{crowleyschick,crowleyschicksteimle}.
\end{enumerate}

Next we discuss results about the homotopy type of automorphism groups of manifolds.
\begin{enumerate}[(i)]
\item In low dimensions $\leq 3$, any topological or PL manifold admits a unique smooth structure and for any smooth manifold $M$ (possibly with non-empty boundary) we have
\[\mr{Diff}_\partial(M) \simeq \mr{PL}_\partial(M) \simeq \mr{Top}_\partial(M).\]

In dimension $n=1$, the only path-connected manifolds are $D^1$ and $S^1$, with $\mr{Diff}_\partial(D^1) \simeq \ast$ and $\mr{Diff}(S^1) \simeq O(2)$. In dimension $n=2$, $B\mr{Diff}_\partial(\Sigma)$ of homologically finite type for all $\Sigma$. This can be proven using analysis (of quadratic differentials \cite{strebel} or harmonic functions \cite{bodigheimer}), or using homotopy theory (\cite{gramain} for the identity components and an induction using arc complexes for the mapping class groups \cite{hatcherarc}). %One then shows that the mapping class group $\pi_0(\mr{Diff}_\partial(\Sigma))$ is of homologically finite type by induction using arc complexes as in \cite{hatcherarc}.

In dimension $n=3$, using \cite{hendrikslaudenbach} one reduces to studying a space related to outer-space (which can be approached using the techniques of \cite{hatchervogtmann}) and the case of prime 3-manifolds. For the latter, the diffeomorphisms groups of nearly all prime 3-manifolds are understood. We will only give a selection of relevant references: the case of Haken manifolds was settled in \cite{hatcher3mfds,ivanovhaken} (which also shows how to reduce from prime 3-manifolds to atoroidal 3-manifolds), and the case of irreducible $M$ with non-empty boundary was settled in \cite{hatchermccullough}. Kalliongis and McCullough showed that for many 3-manifolds $M$,  $\pi_1(\mr{Diff}(M))$ is not finitely generated \cite{kalliongismccullough}.

\item In dimension 4, the best results concern mapping class groups of topological 4-manifolds with good fundamental groups  \cite{quinnisotopy}. For smooth 4-manifolds, the same result holds stably, i.e.\ after taking a connected sum with some number of $S^2 \times S^2$ (such stabilizations are known to be necessary \cite{rubermanisotopy}).

\item In high dimensions $\geq 5$, Farrell and Hsiang also did computations for spheres and aspherical manifolds \cite{farrellhsiang}, which were extended by Burghelea to simply-connected manifolds. In particular, he proved the higher rational homotopy groups of the identity component of $\mr{Diff}(M)$ are finite-dimensional in the concordance stable range \cite{burgheleafiniteness}. For the latest results concerning homeomorphisms of aspherical manifolds, see \cite{ELPUW}. Sullivan proved that $\pi_0$ is commensurable with an arithmetic group if $M$ is simply-connected \cite{sullivaninf}. This fails if the manifold is not simply-connected: for $n \geq 5$, $\pi_0(\mr{Diff}(T^n))$ is not finitely generated \cite{abk2,hatcherconcordance}. This was generalized to higher homotopy groups of $\mr{Diff}(T^n)$ in \cite{hsiangsharpe} (see \cite[Corollary II.4.6]{hatcherwagoner} for a similar example). In general the identity component of $\mr{Diff}(M)$ does not have the homotopy type of a finite CW-complex \cite{hirschnonfinite,lawsonnonfinite,antonelli}.
\end{enumerate}

%These results raise two questions about our results: 
%\begin{question}Corollaries \ref{cor.diff} and \ref{cor.pltop} have the assumption that $M$ is 2-connected. Is this a necessary hypothesis, or does it suffice that $M$ is simply-connected or has finite fundamental group? 
%\end{question} 
%
%\begin{question}\label{que.dim} Are Theorem \ref{thm.main} and its corollaries true in dimensions $5$ and $7$?\end{question}

\subsection{Acknowledgments} We thank S\o ren Galatius and Oscar Randal-Williams for many interesting and helpful conversations, and in particular for suggesting Proposition \ref{prop.collapse} and suggesting to deloop the fiber sequence (\ref{eqn.weissfib}) to obtain (\ref{eqn.weissfib2}). We thank Sam Nariman, Jens Reinhold, Michael Weiss and Allen Hatcher for comments on early versions of this paper, and the referees for helpful suggestions.

\section{Homologically and homotopically finite type spaces} \label{sec.fintype} In this section we recall some basic technical results for proving that homology groups or homotopy groups of spaces are finitely generated. These are elementary and similar to results in the literature, e.g.\ \cite{serrehomotopy,mccleary}, so our proofs will be succinct. %In order, we discuss spaces with finitely generated homology, groups with finitely generated homology, spaces with finitely generated homotopy groups, and section spaces.

\subsection{Homologically finite type spaces} We start with  finiteness conditions that one can impose on the homology groups of a space.

\begin{definition}\begin{itemize}
\item A path-connected space $X$ is said to be of \emph{homologically finite type} if for all $\bZ[\pi_1(X)]$-modules $M$ that are finitely generated as an abelian group, $H_*(X;M)$ is finitely generated in each degree.
\item A space $X$ is said to be of \emph{homologically finite type} if it has finitely many path components and each of these path components is of homologically finite type. We use the notation $X \in \cat{HFin}$.
\end{itemize} \end{definition}

 Using cellular homology, one sees that a CW-complex with finitely many cells in each dimension has finitely generated homology groups. The following lemma generalizes this:

\begin{lemma}\label{lem.cwhfin}A CW complex $X$ with finitely many cells in each dimension is in $\cat{HFin}$.\end{lemma}

\begin{proof}A CW decomposition of $X$ induces a $\pi_1(X)$-equivariant CW decomposition of its universal cover $\tilde{X}$. Letting $\tilde{C}_*$ denote the cellular chains on $\tilde{X}$, a finitely generated complex of free $\bZ[\pi_1(X)]$-modules. Then $H_*(X;M)$ is isomorphic to the homology of the chain complex $\tilde{C}_* \otimes_{\bZ[\pi_1(X)]} M$, which is finitely generated in each degree if $M$ is finitely generated as an abelian group.\end{proof}

\begin{example}\label{exam.handlemfdfin} Compact $n$-dimensional topological manifolds are in $\cat{HFin}$, because they have the homotopy type of finite CW complexes \cite[Corollary 1]{milnorcw}.\end{example}

We will next discuss the behavior of $\cat{HFin}$ under fiber sequences.

\begin{notation}A \emph{fiber sequence with fiber taken over $b \in B$} is a pair of maps $F \overset{i}{\longrightarrow} E \overset{p}{\longrightarrow} B$ and a null-homotopy from the composite $p \circ i \colon F \to B$ to the constant map at $b$, such that the induced map from $F \to \mr{hofib}_b(E \to B)$ is a weak equivalence.\end{notation}

Here the homotopy fiber is defined by replacing $E \to B$ with a fibration $\tilde{p} \colon \tilde{E} \to B$ and taking  $\mr{hofib}_b(E \to B) = \tilde{p}^{-1}(b)$. Thus we have a Serre spectral sequence and long exact sequence of homotopy groups for a fiber sequence. Given a fiber sequence $F \to E \to B$ with $B$ path-connected and $\pi_1(B)$ acting trivially on $H_*(F)$, one may use the Serre spectral sequence to prove that if two of $F$, $E$, $B$ have finitely generated homology in each degree, then so does the third. The following lemma generalizes this:

\begin{lemma}\label{lem.propfinitetypespaces} Let $F \overset{i}{\longrightarrow} E \overset{p}{\longrightarrow} B$ be a fiber sequence with $B$ path-connected.
\begin{enumerate}[\indent (i)]
	\item Bases: if $H_*(F)$ is finitely generated in each degree and $H_*(E;p^*M)$ is finitely generated in each degree for all $\bZ[\pi_1(B)]$-modules $M$ that are finitely generated as abelian groups, then $B \in \cat{HFin}$.
	\item Total spaces: if $B \in \cat{HFin}$, $E$ is path-connected and $H_*(F;i^*M)$ is finitely generated in each degree for all $\bZ[\pi_1(E)]$-modules $M$ that are finitely generated as abelian groups, then $E \in \cat{HFin}$.
	\item Fibers: if $F$ is simply-connected, $H_*(B)$ and $H_*(E)$ are finitely generated in each degree, and $\pi_1(B)$ acts trivially on $H_*(F)$, then $F \in \cat{HFin}$.
\end{enumerate}
\end{lemma}

\begin{proof}\begin{enumerate}[(i)]
\item The $\bZ[\pi_1(B)]$-module $M$ may be pulled back along $p$ to obtain a local coefficient system $p^* M$ on $E$. Applying the equivariant local coefficient Serre spectral sequence of \cite[Theorem 3.1]{moerdijksvensson} with $G = \{e\}$ to the fiber sequence $F \to E \to B$ with local coefficient system $p^* M$ on $E$, we get a spectral sequence
\[E^2_{p,q} = H_p(B;H_q(F;(p \circ i)^* M)) \Longrightarrow H_{p+q}(E;p^* M).\]

Suppose that $M$ is finitely generated as an abelian group. Then, on the one hand, $H_{p+q}(E,p^* M)$ is finitely generated and so are the entries on the $E_\infty$-page. On the other hand, we claim that $H_q(F;(p \circ i)^* M)$ is finitely generated for all $q$. To see this, note that the action of $\pi_1(F)$ on $(p \circ i)^* M$ is trivial and the universal coefficient theorem gives a short exact sequence
\[0 \longto H_q(F) \otimes M \longto H_q(F;(p \circ i)^* M) \longto \mr{Tor}(H_{q-1}(F),M) \longto 0,\]
with both the left and right term finitely generated by our assumptions.

Using these two observations, we prove by induction over $n$ that for all $\bZ[\pi_1(B)]$-modules $M$ that are finitely generated as abelian groups, the groups $H_p(B,M)$ for $p \leq n$ are finitely generated. The initial case is $n=0$: $H_0(B,M)$ is given by the coinvariants $M_{\pi_1(B)}$, which are finitely generated if $M$ is.

For the induction step, we assume the case $n$ and prove the case $n+1$. Since $H_q(F;(p \circ i)^* M)$ is finitely generated for all $q$, the entries $E^2_{p,q} = H_p(B;H_q(F;(p \circ i)^* M))$ on the $E^2$-page of the Serre spectral sequence are  finitely generated for $p \leq n$ by the inductive hypothesis. The entries on further pages are obtained by taking homology,  hence $E^k_{p,q}$ is finitely generated for $p \leq n$ and $k \geq 2$, i.e.\ the first $n$ columns are.

We next claim that for all $k \geq 2$, the entry $E^k_{n+1,0}$ on the first row is finitely generated if and only if $E^2_{n+1,0}$ is finitely generated. The proof is by induction over $k$, the case $k=2$ being obvious. For the induction step, we assume the case $k$ and prove the case $k+1$. As $E^{k+1}_{n+1,0} = \mr{ker}(d^k \colon E^k_{n+1,0} \to E^k_{n+1-k,k-1})$ and the target is finitely generated, $E^{k+1}_{n+1,0}$ is finitely generated if and only if $E^k_{n+1,0}$ is. By the induction hypothesis the latter happens if and only if $E^2_{n+1,0}$ is finitely generated.

The Serre spectral sequence is first-quadrant, so that $E^{n+2}_{n+1,0} \cong E^\infty_{n+1,0}$. The latter is a subquotient of the finitely generated abelian group $H_{n+1}(E;p^* M)$, and thus $E^{n+2}_{n+1,0}$ is finitely generated. Using the previous claim we conclude that $E^2_{n+1,0} = H_{n+1}(B;M)$ is finitely generated as well.
\item This is similar to (i) but easier, using the Serre spectral sequence
\[E^2_{p,q} = H_p(B;H_q(F;i^* M)) \Longrightarrow H_{p+q}(E;M).\]
By the hypothesis all entries on the $E^2$-page are finitely generated. The entries on further pages of the spectral sequence are obtained by taking homology, so $E^k_{p,q}$ is finitely generated for all $k \geq 2$, $p$ and $q$. The entries on the line $p+q$ stabilize after the $(n+1)$-page of the Serre spectral sequence, and thus that $E^\infty_{p,q}$ is finitely generated for all $p$ and $q$. Since $H_{n}(E;M)$ is an iterated extension of the entries $E^\infty_{p,q}$ with $p+q = n$, it also is finitely generated.
\item This is similar to (i), using the Serre spectral sequence $E^2_{p,q} = H_p(B; H_q(F;A)) \Rightarrow H_{p+q}(E;A)$ with $A$ a finitely generated abelian group. We will not give a detailed proof, since this case is not used in the paper.\qedhere
\end{enumerate}
\end{proof}

\subsection{Homologically finite type groups}  We now apply the results of the previous subsection to classifying spaces of groups. By applying Lemma \ref{lem.propfinitetypespaces}(i) to the fiber sequence $G \to EG \to BG$, we see:

\begin{lemma}\label{lem.classifyingspace} If $G$ is a topological group with underlying space in $\cat{HFin}$, then $BG \in \cat{HFin}$. \end{lemma}

\begin{example}\label{exam.bon} Since the orthogonal group $O(n)$ is a compact manifold, Lemma's \ref{lem.cwhfin} and \ref{lem.classifyingspace} imply that $BO(n) \in \cat{HFin}$.\end{example}

The following lemma applies Lemma \ref{lem.propfinitetypespaces} to classifying spaces of discrete groups:

\begin{lemma}\label{lem.groupsprops} The class of groups with classifying space in $\cat{HFin}$ is closed under the following operations:
\begin{enumerate}[\indent (i)]
	\item Quotients: if $1 \to H \to G \to G' \to 1$ is a short exact sequence of groups and $BH,BG \in \cat{HFin}$, then $BG' \in \cat{HFin}$.
	\item Extensions: if $1 \to H \to G \to G' \to 1$ is a short exact sequence of groups and $BH,BG' \in \cat{HFin}$, then $BG \in \cat{HFin}$. This includes products as trivial extensions.
	\item Finite index subgroups: if $G' \subset G$ has finite index then $BG' \in \cat{HFin}$ if and only if $BG \in \cat{HFin}$.
\end{enumerate}
\end{lemma}

\begin{proof}\begin{enumerate}[(i)]
\item This follows from part (i) of Lemma \ref{lem.propfinitetypespaces}, since a short exact sequence  of groups $H \to G \to G'$ induces a fiber sequence of classifying spaces $BH \to BG \to BG'$.
\item Similarly, this follows by applying part (ii) of Lemma \ref{lem.propfinitetypespaces} to $BH \to BG \to BG'$.
\item The direction $\Leftarrow$ follows from part (ii) of Lemma \ref{lem.propfinitetypespaces} applied to $G/G' \to BG' \to BG$.

We first prove the direction $\Rightarrow$ when $G'$ is normal in $G$. As $G/G'$ is finite, from Lemma \ref{lem.classifyingspace} it follows that $B(G/G') \in \cat{HFin}$. The result then follows from part (ii) of this lemma applied to $1 \to G' \to G \to G/G' \to 1$. To reduce the general case to $G'$ normal, use that for any finite index subgroup $G' \subset G$ there is a finite index subgroup $H \subset G'$ such that $H \subset G$ is normal. That $BH \in \cat{HFin}$ follows from the direction $\Leftarrow$.\qedhere
\end{enumerate}\end{proof}

In particular, whether $BG \in \cat{HFin}$ is independent under changing $G$ by finite groups.

\begin{definition}\label{def.diffbyfinite}Two groups $G$, $G'$ \emph{differ by finite groups} (or are said to be \emph{virtually isomorphic}) if there is a finite zigzag of homomorphisms with finite kernel and cokernel:
\[G \longleftarrow G_1 \longto \ldots \longleftarrow G_k \longto G'.\]
\end{definition}

\begin{remark}Two groups $G$, $G'$ are said to be \emph{commensurable} if there is a group $H$ and injective homomorphisms $G \hookleftarrow H \hookrightarrow G'$, each with finite cokernel. Commensurable groups differ by finite groups, and using the fact that an intersection of two finite index subgroups is a finite index subgroup, we see that if $G$ and $G'$ differ by finite groups, they are commensurable.\end{remark}

A second class of groups with classifying space in $\cat{HFin}$: arithmetic groups. We use the definition in \cite{serrearithmetic} (in particular, we do \emph{not} assume the $\bQ$-algebraic group $G$ is reductive).

\begin{definition}\begin{itemize} \item A subgroup $\Gamma$ of a $\bQ$-algebraic group $G \subset GL_n(\bQ)$ is \emph{arithmetic} if the intersection $\Gamma \cap GL_n(\bZ)$ has finite index in both $\Gamma$ and $G \cap GL_n(\bZ)$.
\item A group $\Gamma$ is \emph{arithmetic} if it can be embedded in a $\bQ$-algebraic group $G$ as an arithmetic subgroup.\end{itemize}\end{definition}

Examples of arithmetic groups include all finite groups, and all finitely generated abelian groups, see Section 1.2 of \cite{serrearithmetic}. %By definition arithmetic group are closed under passing to finite index subgroups.

\begin{theorem}[Borel-Serre]\label{thm.arithmetic} If $G$ is an arithmetic group, then $BG \in \cat{HFin}$.
\end{theorem}

\begin{proof} By property (5) of Section 1.3 of \cite{serrearithmetic}, every torsion-free arithmetic group $\Gamma$ has a $B\Gamma$ that is a finite CW complex and using Lemma \ref{lem.cwhfin} we conclude $B\Gamma \in \cat{HFin}$. By property (4) of Section 1.3 of \cite{serrearithmetic} every arithmetic group $G$ has a finite index torsion-free subgroup $\Gamma$, so from part (iii) of Lemma \ref{lem.groupsprops} it follows that any arithmetic group $G$ has $BG \in \cat{HFin}$.
\end{proof}

\subsection{Homotopically finite type spaces} Alternatively we can impose finiteness conditions on the homotopy groups of a space. The following distinction will be useful.

\begin{definition}\begin{itemize}\item We say $X$ is of \emph{homotopically finite type} if it has finitely many path components and each of its path components has the property that $B\pi_1 \in \cat{HFin}$ and $\pi_i$ is finitely generated for $i \geq 2$. We use the notation $X \in \cat{\Pi Fin}$.
\item We say $X$ is of \emph{finite type} if it has finitely many path components and each of its path components has the property that $\pi_1$ is finite and $\pi_i$ is finitely generated for $i \geq 2$.  We use the notation $X \in \cat{Fin}$.\end{itemize}
\end{definition}

By Lemma \ref{lem.classifyingspace}, finite groups have classifying spaces in $\cat{HFin}$, so that $\cat{Fin} \subset \cat{\Pi Fin}$. Using Postnikov towers, we shall prove that $\cat{\Pi Fin} \subset \cat{HFin}$ and that $\cat{HFin} \cap \{\text{$\pi_1$ is finite}\} \subset \cat{Fin}$. If $X$ is a path-connected space, it has a \emph{Postnikov tower with $n$th stage $P_n(X)$} \cite[Section VI.2]{goerssjardine}. This has the following properties: firstly, the homotopy groups of $P_n(X)$ are given by
\[\pi_i( P_n(X)) = \begin{cases} \pi_i(X) & \text{if $i \leq n$,}\\
0 & \text{if $i>n$,}\end{cases}\] 
and secondly there are fiber sequences
\begin{equation}\label{eqn.postnikov} K(\pi_n (X),n) \longto P_n (X) \longto P_{n-1} (X).\end{equation}

\begin{lemma}\label{lem.emhfin}If $A$ is a finitely generated abelian group and $n \geq 1$, then $K(A,n) \in \cat{HFin}$.\end{lemma}

\begin{proof}The proof is by induction over $n$. In the initial case $n=1$, we use that $K(A,1) = BA$ is well-known to have a CW model with finitely cells in each dimension, and hence is in $\cat{HFin}$ by Lemma \ref{lem.cwhfin}. For the induction step, apply Lemma \ref{lem.propfinitetypespaces}(i) to the fiber sequence $K(A,n-1) \to * \to K(A,n)$.\end{proof} 

\begin{lemma}\label{lem.pfinhfin} If $X \in \cat{\Pi Fin}$, then $X \in \cat{HFin}$.\end{lemma}

\begin{proof}Without loss of generality $X$ is path-connected. Because $P_0(X) \simeq *$, $P_1(X) \simeq K(\pi_1 (X),1) \simeq B\pi_1(X) \in \cat{HFin}$ by assumption. To finish the proof, apply Lemma \ref{lem.emhfin} and Lemma \ref{lem.propfinitetypespaces}(ii) inductively to the fiber sequence (\ref{eqn.postnikov}).\end{proof}

We continue with the proof that $\cat{HFin} \cap \{\text{$\pi_1$ is finite}\} \subset \cat{Fin}$.

\begin{lemma}\label{lem.finitefinprep} If $X \in \cat{HFin}$ and each component is simply-connected, then $X \in \cat{Fin}$.\end{lemma}

\begin{proof} Without loss of generality $X$ is path-connected. Our proof is by induction over $n$ of the statement that $\pi_i$ with $i \leq n$ is finitely generated. The initial case $n = 1$ follows from the fact that $X$ is simply-connected. 
	
Suppose we have proven the case $n$, then we will prove the case $n+1$. We claim that $H_{*}(P_{n+1}(X))$ is finitely generated for $* \leq n+2$. We start with $* = n+1,n+2$: the map $X \to P_{n+1}(X)$ is $(n+2)$-connected (i.e.\ an isomorphism on $\pi_i$ for $i \leq n+1$ and a surjection for $* = n+2$). Thus $H_{*}(X)$ surjects onto $H_{*}(P_{n+1}(X))$ for $* = n+1,n+2$ and the latter are finitely generated. For $* \leq n$, note that $P_{n+1}(X) \to P_n(X)$ is $(n+1)$-connected so $H_*(P_{n+1}(X)) \cong H_*(P_{n}(X))$ for $* \leq n$. By the inductive hypothesis $P_{n}(X)$ is simply-connected with finitely generated higher homotopy groups, so in $\cat{HFin}$ by Lemma \ref{lem.pfinhfin}. The long exact sequence for homology of a pair then tells us that $H_{n+2}(P_n(X),P_{n+1}(X))$ is also finitely generated. The relative Hurewicz theorem completes the proof
\[H_{n+2}(P_n(X),P_{n+1}(X)) \cong  \pi_{n+2}(P_n(X),P_{n+1}(X)) \cong \pi_{n+1}(X).\qedhere\]
\end{proof}

\begin{lemma}\label{lem.finitefin} If $X \in \cat{HFin}$ and $\pi_1$ of each component is finite, then $X \in \cat{Fin}$.\end{lemma}

\begin{proof} Without loss of generality $X$ is path-connected. Let $\tilde{X}$ denote the universal cover of $X$. We start by proving $\tilde{X} \in \cat{HFin}$. We have $H_*(\tilde{X};M) \cong H_*(X;\bZ[\pi_1(X)] \otimes M)$, so $\tilde{X} \in \cat{HFin}$ follows from the assumptions that $X \in \cat{HFin}$ and that $\pi_1(X)$ is finite. The higher homotopy groups of $\tilde{X}$ are those of $X$, so it suffices to show that $\pi_i(\tilde{X})$ is finitely generated, which follows from Lemma \ref{lem.finitefinprep}.\end{proof}

%There is a fiber sequence
%\[\tilde{X} \to X \to B\pi_1(X)\]
%Applying the total space part (ii) of Lemma \ref{lem.propfinitetypespaces}, together with the fact $X \in \cat{HFin}$, $\pi_1(X)$ is finite and thus $B\pi_1(X) \in \cat{HFin}$ by part (i) of Lemma \ref{lem.groupshfin}, we conclude that $\tilde{X} \in \cat{HFin}$. 

%Now consider the fiber sequence
%\[P_{n+1} X \to P_{n} X \to K(\pi_{n+1} X,{n+2})\]
%The argument for part (i) of Lemma \ref{lem.propfinitetypespaces} goes through to give that $H_{n+2}(K(\pi_{n+1} X,{n+2}))$ is finitely generated. By the Hurewicz theorem it is isomorphic to $\pi_{n+1}(X)$, and thus the latter is finitely generated.

If $X$ is path-connected, let $X\langle n \rangle$ denote the \emph{$n$-connective cover}, the homotopy fiber of $X \to P_n(X)$. This is well-defined  up to homotopy, and has the following property:
\[\pi_i( X\langle n \rangle) = \begin{cases} 0 & \text{if $i \leq n$,}\\
\pi_i (X) & \text{if $i>n$.}\end{cases}\]

\begin{lemma}\label{lem.connectivecover} If $X$ is path-connected, $\pi_1(X)$ is finite and $X \in \cat{HFin}$, then $X\langle n \rangle \in \cat{Fin}$ for all $n \geq 0$. \end{lemma}

\begin{proof}By Lemma \ref{lem.finitefin} we can replace $\cat{HFin}$ by $\cat{Fin}$ and in particular the homotopy groups of $X$ are finitely generated. Since $\pi_i(X \langle n \rangle)$ is given by $0$ for $i \leq n$ and $\pi_i(X)$ for $i > n$, it is finitely generated.
\end{proof}

\begin{example}\label{exam.bo2nn} Since $\pi_0(O(n)) \cong \bZ/2\bZ$, from Example \ref{exam.bon} and Lemma \ref{lem.finitefin} we conclude that $BO(n) \in \cat{Fin}$. Using Lemma \ref{lem.connectivecover} we conclude that $BO(2n)\langle n \rangle \in \cat{Fin}$ as well.\end{example}

For $\cat{\Pi Fin}$ or $\cat{Fin}$, the long exact sequence of homotopy groups is useful: if $p \colon E \to B$ be a Serre fibration with base point $e_0 \in E$ and $F$ the fiber of $p$ over $p(e_0)$, then there is a long exact sequence 
\begin{equation}\label{eqn.les}\cdots \longto \pi_1(F,e_0) \longto \pi_1(E,e_0) \longto \pi_1(B,p(e_0)) \longto \pi_0(F) \longto \pi_0(E) \longto \pi_0(B),\end{equation}
where the right-most three entries are pointed sets, with path components containing $e_0$ and $p(e_0)$ providing the base points, the next three entries are groups, and the remaining entries are abelian groups.

\begin{lemma}\label{lem.les} Let $p \colon E \to B$ be a Serre fibration, $e_0 \in E$ a base point and $F = p^{-1}(p(e_0))$. Let $F_0$, $E_0$ and $B_0$ denote the path components of $F$, $E$ and $B$ containing $e_0$, $e_0$ and $p(e_0)$ respectively. Then the following holds for $\pi_i$ for $i \geq 2$:
\begin{enumerate}[\indent (i)]
\item Bases: if $F_0$ has finitely generated $\pi_i$ for $i \geq 2$ and $E_0$ has finitely generated $\pi_i$ for $i \geq 3$, then $B_0$ has finitely generated $\pi_i$ for $i \geq 3$. $\pi_2(B_0)$ is finitely generated if additionally $\pi_2(E_0)$ is finitely generated and either 
\begin{enumerate}[(a)]
\item $\pi_1(F_0)$ is finite,
\item $\pi_1(F_0)$ is a finitely generated abelian group, or
\item $\pi_1(F_0)$ is finitely generated and $\pi_1(E_0)$ is finite.
\end{enumerate}
\item Total spaces: if $F_0$, $B_0$ have finitely generated $\pi_i$ for $i \geq 2$ then so does $E_0$.
\item Fibers: if $E_0$ has finitely generated $\pi_i$ for $i \geq 2$ and $B_0$ has finitely generated $\pi_i$ for $i \geq 3$, then $F_0$ has finitely generated $\pi_i$ for $i \geq 2$.
\end{enumerate}

Furthermore, the following holds for $\pi_1$:
\begin{enumerate}[\indent (i')]
\item Bases: if $\pi_1(E_0)$ is finitely generated and $\pi_0(F)$ is finite, then $\pi_1(B_0)$ is finitely generated.
\item Total spaces: $\pi_1(E_0)$ is finitely generated, if $\pi_1(F_0)$ and $\pi_1(B_0)$ are finitely generated and either 
\begin{enumerate}[(a)]
\item $\pi_1(B_0)$ is finite,
\item $\pi_1(B_0)$ is abelian, or
\item $\pi_0(F)$ is finite.
%\item $E_0 \to B_0$ has a section,
\end{enumerate} 
\item Fibers: $\pi_1(F_0)$ is finitely generated, if $\pi_2(B_0)$ and $\pi_1(E_0)$ are finitely generated and either 
\begin{enumerate}[(a)]
\item $\pi_1(E_0)$ is finite,
\item $\pi_1(E_0)$ is abelian, or
\item $\pi_1(B_0)$ is finite.
\end{enumerate}
\end{enumerate}
\end{lemma}

\begin{proof}Since $\pi_i$ for $i \geq 1$ depends only the component containing the base point, (\ref{eqn.les}) equals
\[\cdots \longto \pi_2(B_0,p(e_0)) \longto \pi_1(F_0,e_0) \longto \pi_1(E_0,e_0) \longto \pi_1(B_0,p(e_0)) \to \pi_0(F).\]

Parts (ii), (iii) and the first claim in (i) follow from this and the fact that in the category of abelian groups, the class of finitely generated abelian groups is closed under taking subgroups, quotients and extensions. This argument fails for the second claim in part (i), i.e.\ for  $\pi_2(B_0)$, because it only says that $\pi_2(B_0)$ is an extension of a subgroup of $\pi_1(F_0)$ by a quotient of $\pi_2(E_0)$, but subgroups of finitely generated non-abelian groups need not be finitely generated. The remainder of (i) gives conditions to guarantee this problem does not occur, using the following facts: 
\begin{enumerate}[(a)]
\item a subgroup of a finite group is finite,
\item a subgroup of a finitely generated abelian group is a finitely generated abelian group,
\item a group is finitely generated if and only if a finite index subgroup is finitely generated.
\end{enumerate}
Parts (i'), (ii') and (iii') are obtained by applying these same facts and using that finitely generated groups are closed under extensions.\end{proof}

\subsection{Section spaces} We continue with examples of spaces that can be proven to be in $\cat{Fin}$ or $\cat{\Pi Fin}$. If $E \to B$ is a Serre fibration, let $\Gamma(E,B)$ denote the space of sections in the compact-open topology. Given a section $s$ and a subspace $A \subset B$, then $\Gamma(E,B;A) \subset \Gamma(E,B)$ denotes the subspace of sections that are equal to $s$ on $A$.

\begin{lemma}\label{lem.sectionspace} Let $E \to B$ be a Serre fibration with fiber $F$ and section $s \colon B \to E$. Suppose that $B$ is a path-connected finite CW-complex, $A \subset B$ is a non-empty subcomplex of $B$, and each component of $F$ has finitely generated homotopy groups. Then each component of $\Gamma(E,B;A)$ has finitely generated homotopy groups. The condition that $A \neq \varnothing$ can be dropped if $\pi_1$ of each component of $F$ is finite or abelian, or $F$ has finitely many components.\end{lemma}

\begin{proof}We prove the first part by induction over the number $k$ of cells in $B$ that are not in $A$. In the initial case $k=0$, $\Gamma(E,B;B) = \{s\}$. For the induction step, suppose we have proven the case $k$, then we will prove the case $k+1$. That is, we have $A' = A \cup_{S^{d-1}} D^d$ where there are $k$ cells of $B$ not in $A'$. The induction hypothesis says the claim is true for $A'$, and we want to prove it for $A$.

There are two cases, the first of which is $d \geq 1$. Restriction of sections to $D^d$ fits into a fiber sequence
\[\Gamma(E,B;A \cup D^d) \longto \Gamma(E,B;A) \longto \Gamma(E|_{D^d},D^d;S^{d-1}),\]
with fiber taken over $s|_{D^d} \in \Gamma(E|_{D^d},D^d;S^{d-1})$.  We apply the induction hypothesis to the fiber, which says each component of $\Gamma(E,B;A \cup D^d)$ has finitely generated $\pi_i$ for $i \geq 1$. The base $\Gamma(E|_{D^d},D^d;S^{d-1})$ is equivalent to $\Omega^d F$, whose components have finitely generated $\pi_i$ for $i \geq 1$ by our hypothesis. The result now follows from the parts (ii) and (ii') of Lemma \ref{lem.les}. For part (ii') we use that condition (b) is satisfied: $\pi_1(\Omega^d F)$ is abelian if $d \geq 1$.

The second case is $d=0$, so that $A' = A \sqcup D^0$. Since $B$ is path-connected, by adding some $1$-cells to $A'$ we can obtain a subcomplex $A''$ of $B$ containing $A'$ such that the inclusion $A \hookrightarrow A''$ is a homotopy equivalence. This implies that the inclusion map $\Gamma(E,B;A'') \to \Gamma(E,B,A)$ is a weak equivalence, which completes the proof since $A''$ has more cells than $A$ and hence we can apply the inductive hypothesis. 

For the second claim, we consider the fiber sequence
\[\Gamma(E,B;D^0) \longto \Gamma(E,B) \longto E|_{D^0},\]
with fiber taken over $s \in E|_{D_0}$. By the first part, each component of $\Gamma(E,B;D^0)$ has finitely generated homotopy groups. Furthermore, the base is equivalent to $F$. Thus the result follows part (ii) and (ii') of Lemma \ref{lem.les}, using that condition (a), (b) or (c) holds by assumption.
\end{proof}
	
\section{Self-embeddings}\label{sec.selfemb} In this section we  study spaces of self-embeddings. Let $M$ be a smooth $n$-dimensional connected manifold with non-empty boundary $\partial M$ and an embedded disk $D^{n-1} \subset \partial M$. Let $\mr{Emb}_{\half\partial}(M)$ be the space of embeddings $M \hookrightarrow M$ that are the identity on $ \partial M \backslash \mr{int}(D^{n-1})$. They do \emph{not} need to take the entire boundary to itself. There is an inclusion $\mr{Diff}_\partial(M) \hookrightarrow \smash{\mr{Emb}_{\half\partial}}(M)$ and its image is contained in the following subspace:

\begin{definition}We let $\smash{\mr{Emb}^{\cong}_{\half\partial}}(M)$ denote the subspace of $\mr{Emb}_{\half\partial}(M)$ consisting of those embeddings $e$ that are isotopic through embeddings $e_t$ that are the identity on $\partial M \backslash \mr{int}(D^{n-1})$ to a diffeomorphism $e_1$ of $M$ which is the identity on $\partial M$.\end{definition}

That is, $\smash{\mr{Emb}^{\cong}_{\half\partial}}(M)$ is the union of the path components of $\smash{\mr{Emb}_{\half\partial}}(M)$ in the image of $\pi_0(\mr{Diff}_\partial(M))$. Composition gives $\smash{\mr{Emb}^{\cong}_{\half\partial}}(M)$ the structure of a topological monoid, and hence its connected components are a monoid. However, since every path component is represented by a diffeomorphism and hence has an inverse, $\pi_0 (\smash{\mr{Emb}^{\cong}_{\half \partial}}(M))$ is in fact a group. The main result of this section is the following:

\begin{theorem}\label{thm.emb} If $n \geq 6$, $M$ is $2$-connected and $\partial M = S^{n-1}$, then $B\mr{Emb}^{\cong}_{\half\partial}(M) \in \cat{\Pi Fin}$ and hence $\cat{HFin}$.\end{theorem}

\begin{proof}Proposition \ref{prop.embpi0hfin} says that the classifying space of the group of components is in $\cat{HFin}$ and Proposition \ref{prop.embidhfin}  that the identity component is in $\cat{\Pi Fin}$; thus $B\mr{Emb}^{\cong}_{\half\partial}(M) \in \cat{\Pi Fin}$.\end{proof}

\subsection{Variations of the definitions} Depending on the result to be proven, it may useful to work with alternative models for $\mr{Emb}_{\half \partial}(M)$. 

\subsubsection{Embeddings rel boundary} In an embedding space we may change the domain or target up to isotopy equivalence rel boundary without changing its homotopy type. Let $M$ and $N$ be $n$-dimensional manifolds with codimension zero submanifold $K$ contained in both $\partial M$ and $\partial N$, then we define $\mr{Emb}_K(M,N)$ to be the space of embeddings $M \hookrightarrow N$ that are the identity on $K$. Thus $\smash{\mr{Emb}^{\cong}_{\half\partial}}(M)$ is the case $M = N$ and $K = \partial M \setminus \mr{int}(D^{n-1})$. If $M'$ is another $n$-dimensional manifold such that $K$ is a codimension zero submanifold of $\partial M$, then we say that $M'$ is \emph{isotopy equivalent to $M$ rel $K$} if there are embeddings $M' \hookrightarrow M$ and $M \hookrightarrow M'$ rel $K$ which are inverse up to isotopy rel $K$.

\begin{lemma}If $M'$ is isotopy equivalent to $M$ rel $K$, and $N'$ is isotopy equivalent to $N$ rel $K$, then $\mr{Emb}_K(M,N) \simeq \mr{Emb}_{K}(M',N')$.\end{lemma}

\begin{proof}We give the proof for replacing $M$ by $M'$, leaving the similar proof for replacing $N$ by $N'$ to the reader. Let $f \colon M \hookrightarrow M'$ and $g \colon M' \hookrightarrow M$ denote the embeddings, then precomposition gives two maps
	\[f^* \colon \mr{Emb}_K(M',N) \longto \mr{Emb}_K(M,N) \qquad \text{and} \qquad g^* \colon \mr{Emb}_K(M,N) \longto \mr{Emb}_K(M',N),\]
and the isotopies $g \circ f \sim \mr{id}_{M}$ and $f \circ g \sim \mr{id}_{M'}$ induce homotopies $f^* \circ g^* \sim \mr{id}_{\mr{Emb}_K(M,N)}$ and $g^* \circ f^* \sim \mr{id}_{\mr{Emb}_K(M',N)}$.	
\end{proof}

Using this we may replace $\mr{Emb}_{\half \partial}(M)$ by a space of embeddings rel the entire boundary: let $M^\ast \coloneqq M \setminus \mr{int}(D^{n-1})$, which has boundary $\partial M^\ast = \partial M \setminus \mr{int}(D^{n-1})$. 

\begin{lemma}$M$ and $M^\ast$ are isotopy equivalent rel $\partial M \setminus \mr{int}(D^{n-1})$.\end{lemma}

\begin{proof}One of the embeddings is given by the inclusion $M^\ast \hookrightarrow M$. For the other we pick a collar $\partial M \times [0,1) \hookrightarrow M$, a smooth function $\lambda \colon \partial M \to [0,1)$ which equals $0$ on $\partial M \setminus \mr{int}(D^{n-1})$ and is strictly positive on $\mr{int}(D^{n-1})$, and a family $\rho_s \colon [0,1) \to [0,1)$ for $s \in [0,1)$ of embeddings such that $\rho_s([0,1)) = [s,1)$ and $\rho_s$ is the identity near $1$. Then there is an embedding $M \hookrightarrow M^\ast$ given by
\[m \longmapsto \begin{cases} (m',\rho_{\lambda(m')}(t)) & \text{if $m = (m',t') \in \partial M \times [0,1)$,}\\
m & \text{otherwise.}\end{cases}\]
By linearly interpolating $\lambda$ to the function that is constant equal to $0$ one obtains the desired isotopies.
\end{proof}

Using the previous lemma we conclude that:

\begin{lemma}\label{lem.embweqbdy} There is a weak equivalence $\mr{Emb}_{\half \partial}(M) \simeq \mr{Emb}_{\partial}(M^\ast,M^\ast)$.\end{lemma}

\subsubsection{Adding semi-infinite collars} 
Let $M_\propto$ be the non-compact manifold  (the symbol $\propto$ is meant to evoke a half-open $\infty$) obtained by gluing a semi-infinite collar to $\partial M$:
\[M_\propto \coloneqq M \cup_{(\partial M \times \{0\})} (\partial M \times [0,\infty)).\]
It has a submanifold with corners
\[M'_\propto \coloneqq M \cup_{(\partial M \backslash \mr{int}(D^{n-1})) \times \{0\}} (\partial M \backslash \mr{int}(D^{n-1})) \times [0,\infty).\] 

We describe two variants of $\mr{Emb}_{\half \partial}(M)$
\[\mr{Emb}^{\cong}_{\half\partial}(M) \hookrightarrow \mr{Emb}^{\cong,\mr{im}}_{\half\partial}(M'_\propto,M_\propto) \hookrightarrow \mr{Emb}^{\cong}_{\half\partial}(M'_\propto,M_\propto)\]
such that the first inclusion is a homeomorphism and the second is a weak equivalence.

\begin{figure}
\begin{tikzpicture}
	\draw (5,-0.6) -- (2,-0.6) to[out=180,in=0] (0,-1) to[out=180,in=-90] (-1,0) to[out=90,in=180] (0,1) to[out=0,in=180] (2,0.6) -- (5,0.6);
	\draw (.4,0) to[out=-90,in=-90] (-.4,0);
	\draw (.3,-0.15) to[out=90,in=90] (-.3,-0.15);
	\draw [dotted] (2,0.6) to[out=0,in=0] (2,-0.6);
	\draw [densely dotted] (2,0.6) to[out=180,in=180] (2,-0.6);
	\node at (0,-1.3) {$M$};
	\node at (2,-.95) {$\partial M$};
	\node at (4,0) {$\partial M \times [0,\infty)$};
	\node at (2,1.5) {$M_\propto$};
	\draw (1.55,.2) -- (1.8,.2);
	\draw (1.55,-.2) -- (1.8,-.2);
	\node at (1.7,0) [left] {\tiny $D^{n-1}$};
	
	\begin{scope}[xshift=8cm]
	\draw (5,-0.6) --(2,-0.6) to[out=180,in=0] (0,-1) to[out=180,in=-90] (-1,0) to[out=90,in=180] (0,1) to[out=0,in=180] (2,0.6) -- (5,0.6);
%	\draw [dotted] (2,-0.6) -- (5,-0.6);
%	\draw [dotted] (2,0.6) -- (5,0.6);
	\draw [densely dotted](2,-0.6) to[out=180,in=-80] (1.65,-.2);
	\draw [densely dotted](2,0.6) to[out=180,in=80] (1.65,.2);
	\draw (5,-.2) -- (1.65,-.2) to[out=100,in=-100] (1.65,.2) --(5,.2);
	\draw (.4,0) to[out=-90,in=-90] (-.4,0);
	\draw (.3,-0.15) to[out=90,in=90] (-.3,-0.15);
	\draw [dotted] (2,0.6) to[out=0,in=0] (2,-0.6);
	\node at (0,-1.3) {$M$};
	\node at (2,-.95) {$\partial M$};
	\node at (2,1.5) {$M'_\propto$};
	\end{scope}
\end{tikzpicture}
\label{fig.mpropto}
\caption{The manifolds $M'_\propto \subset M_\propto$.}
\end{figure}

\begin{definition}\label{def.embmpm}Let $\mr{Emb}^{\cong}_{\half\partial}(M'_\propto,M_\propto)$ be the space of embeddings $e \colon M'_\propto \hookrightarrow M_\propto$ that are the identity on $(\partial M \backslash \mr{int}(D^{n-1})) \times [0,\infty)$ and are isotopic through embeddings $e_t$ that are the identity on $(\partial M \backslash \mr{int}(D^{n-1})) \times [0,\infty)$ to an embedding $e_1$ which is the identity on $M'_\propto \setminus \mr{int}(M)$.\end{definition}

\begin{remark}An equivalent definition of $\mr{Emb}^{\cong}_{\half\partial}(M'_\propto,M_\propto)$ uses complements instead of isotopies; it equals those embeddings $e \colon M'_\propto \hookrightarrow M_\propto$ that are the identity on $(\partial M \backslash \mr{int}(D^{n-1})) \times [0,\infty)$ and have the property that the closure of the complement of their image admits a compactly-supported diffeomorphism to $D^{n-1} \times [0,\infty)$ rel boundary.\end{remark}

%Let us consider an embedding $e$ of $M'_\propto$ into $M_\propto$ that is the identity on $(\partial M \backslash \mr{int}(D^{n-1})) \times [0,\infty)$. The closure of the complement of $\mr{im}(e)$ will be a manifold with corners, denoted $C(e)$. Its boundary is identified  with the boundary of $D^{n-1} \times [0,\infty)$ by $e$. Thus we can demand that $C(e)$ is diffeomorphic to $D^{n-1} \times [0,\infty)$ by a diffeomorphism $\phi$ that preserves the identification of the boundary. Furthermore, since $\mr{im}(e)|_{M}$ is compact, it also makes sense for this diffeomorphism to have compact support.

\begin{definition}\label{def.embmpmim} Let $\mr{Emb}^{\cong,\mr{im}}_{\half\partial}(M'_\propto,M_\propto)$ be the subspace of $\mr{Emb}^{\cong}_{\half\partial}(M'_\propto,M_\propto)$ consisting of embeddings $e$ that satisfy $e(M) \subset M$.
\end{definition}

\begin{lemma}\label{lem.embimageweq} The inclusion $\mr{Emb}^{\cong,\mr{im}}_{\half\partial}(M'_\propto,M_\propto) \hookrightarrow \mr{Emb}^{\cong}_{\half\partial}(M'_\propto,M_\propto)$ is a weak equivalence.\end{lemma}

\begin{proof}Suppose we are given a commutative diagram
	\[\begin{tikzcd}S^i \dar \rar & \mr{Emb}^{\cong,\mr{im}}_{\half\partial}(M'_\propto,M_\propto) \dar \\
	D^{i+1} \rar[swap]{f} \arrow[dotted]{ru}& \mr{Emb}^{\cong}_{\half\partial}(M'_\propto,M_\propto),\end{tikzcd}\]
	then we must provide a dotted lift making the diagram commute, possibly after changing it through a homotopy of commutative diagrams.
	
	It suffices to push the image of $M$ under $f_s$ out of $\mr{int}(D^{n-1}) \times (0,\infty)$ for each $s \in D^{i+1}$. To do this, let $\mr{Emb}^{\cong}_{\half\partial}(M'_\propto,M_\propto)'$ be the subspace of $\mr{Emb}^{\cong}_{\half\partial}(M'_\propto,M_\propto)$ consisting of those $e$ that satisfies the property that $e(M'_\propto) \subset M'_\propto \cup (D^{n-1}\setminus \frac{1}{2}D^{n-1}) \times [0,1)$. The inclusion $\mr{Emb}^{\cong}_{\half\partial}(M'_\propto,M_\propto)' \hookrightarrow \mr{Emb}^{\cong}_{\half\partial}(M'_\propto,M_\propto)$ is weak equivalence by a collar-sliding argument, so we may assume $f_s \in \mr{Emb}^{\cong}_{\half\partial}(M'_\propto,M_\propto)'$.
	
%	Now pick a family of   $D^{i+1}$ is compact, there exists a smooth function $\gamma \colon D^{n-1} \to [0,\infty)$ such that
%	\begin{enumerate}[\indent (1)]
%		\item $\gamma|_{\partial D^{n-1}}$ is constant equal to $0$, and 
%		\item $f_s(M) \subset M \cup \{(d,t) \in D^{n-1} \times [0,\infty) \mid t \leq \gamma(d) \}$ for each $s \in D^{i+1}$.
%	\end{enumerate}
	 Using a collar $\partial M \times [0,1) \hookrightarrow M$, we construct a family of diffeomorphisms $\psi_r \colon M_\propto \to M_\propto$ for $r\in [0,1]$ with the following properties:  
	\begin{enumerate}[\indent (i)]
		\item $\psi_r$ is the identity on $M'_\propto$
		\item $\psi_0 = \mr{id}$,
		\item $\psi_r(M \cup (\frac{1}{2} D^{n-1} \times [0,t])) \subset M$.
	\end{enumerate} 
	Then the formula 
	\begin{align*}D^{i+1} \times [0,1] &\longto \mr{Emb}^{\cong}_{\half\partial}(M'_\propto,M_\propto)' \\
	(s,r) &\longmapsto \psi_r f_s\end{align*} gives us a homotopy which starts at $r=0$ with $f_s$ by (ii), and sends $D^{i+1} \times \{1\} \cup S^i \times [0,1]$ into $\mr{Emb}^{\cong,\mr{im}}_{\half\partial}(M'_\propto,M_\propto)$ by (i) and (iii). After this homotopy through commutative diagrams an evident lift exists.
\end{proof}

We can extend an embedding $e \in \mr{Emb}_{\half \partial}(M)$ to an embedding $\bar{e} \in \mr{Emb}_{\half \partial}(M'_\propto,M_\propto)$ by defining it to be the identity on $(\partial M \backslash \mr{int}(D^{n-1})) \times [0,\infty)$. 

\begin{lemma}\label{lem.embextweq}Extension induces a homeomorphism
	\[\mr{Emb}^{\cong}_{\half\partial}(M) \hookrightarrow \mr{Emb}^{\cong,\mr{id}}_{\half\partial}(M'_\propto,M_\propto).\]
\end{lemma}

\begin{proof}We need to verify that $e \colon M \hookrightarrow M$ is isotopic to a diffeomorphism if and only if $\bar{e}$ is isotopic to an embedding that is the identity on $M'_\propto \setminus \mr{int}(M)$. The direction $\Rightarrow$ is obvious. For $\Leftarrow$, given $\bar{e}_t \colon M'_\propto \hookrightarrow M_\propto$ we may use the proof of Lemma \ref{lem.embimageweq} to modify it to an isotopy which always maps $M$ into $M$ and thus induces the desired isotopy of $e$.
\end{proof}

\subsection{The group of path components} We start by studying the group of path components using a result of Sullivan.

\begin{proposition}\label{prop.embpi0hfin} Let  $n \geq 5$ and suppose we are given an $n$-dimensional oriented smooth manifold $M$ with finite fundamental group and $\partial M = S^{n-1}$. Then $B\pi_0( \diff_\partial(M))$ and $B\pi_0(\mr{Emb}^{\cong}_{\half\partial}(M))$ are in $\cat{HFin}$.\end{proposition}

\begin{proof} Let $N = M \cup_{S^{n-1}} D^n$ and $\mr{Emb}^+(D^n,M)$ denote the orientation-preserving embeddings. Then there is a fiber sequence
\[\mr{Diff}_\partial(M) \longto \mr{Diff}^+(N) \longto \mr{Emb}^+(D^n,N)\]
with fiber taken over the standard embedding $D^n \hookrightarrow N$. Using the hypothesis on $M$, the space $\mr{Emb}^+(D^n,N)$ is path-connected and has finite $\pi_1$. Thus $\pi_0(\mr{Diff}_\partial(M))$ differs by finite groups from $\pi_0(\mr{Diff}^+(N))$. Sullivan proved that for closed oriented 1-connected $N$ of dimension $\geq 5$, $\pi_0(\mr{Diff}^+(N))$ is commensurable with an arithmetic group \cite[Theorem 13.3]{sullivaninf} and Triantafillou generalized this to oriented manifolds with finite fundamental groups \cite{trianta}. We conclude $\pi_0(\mr{Diff}_\partial(M))$ differs by finite groups from an arithmetic group, and thus Theorem \ref{thm.arithmetic} implies the first part.
 
For $\pi_0(\mr{Emb}^{\cong}_{\half\partial}(M))$, note that Theorem \ref{thm.weissfibersequence} implies there is an exact sequence of groups
 \[\pi_0(\mr{Diff}_\partial(D^{n})) \longto \pi_0 (\mr{Diff}_\partial(M)) \longto \pi_0 (\mr{Emb}^{\cong}_{\half\partial}(M)) \longto 1.\]
Since $\pi_0(\mr{Diff}_\partial(D^{n})) \cong \Theta_{n+1}$ is finite, $\pi_0 (\mr{Diff}_\partial(M))$ and $\pi_0 (\mr{Emb}^{\cong}_{\half\partial}(M))$ differ by finite groups and the first part suffices.
\end{proof}

\subsection{Embedding calculus} Embedding calculus is the ``pointillistic'' study of embeddings \cite[Remark 4.5.4]{weissdalian}. An embedding is an immersion that is injective when evaluated on any finite subset. If we replace this condition with homotopy-theoretic data, we get an object accessible to homotopy theory. The multiple disjunction theory results of \cite{goodwillieklein} imply that the space of these homotopy-theoretic alternatives is weakly equivalent to the space of embeddings when the codimension is at least 3. Our references for embedding calculus are \cite{weissembeddings,weisspedrosheaves} (see \cite{weisspedro} for a different perspective). 

\subsubsection{Manifolds without boundary}
Fix two manifolds $M$ and $N$. Then $\mr{Emb}(M,N)$ is the value on $M$ of a continuous functor $\mr{Emb}(-,N) \colon \cat{Mfd}_n^\mr{op} \to \cat{Top}$. Here $\cat{Mfd}_n$ is the topological category with objects given by $n$-dimensional smooth manifolds and morphism spaces given by spaces of embeddings, and $\cat{Top}$ is the topological category with objects given by CGWH spaces and morphism spaces given by spaces of continuous maps. The category $\cat{Mfd}_n$ admits a collection of Grothendieck topologies $\cJ_k$ for $k \geq 1$; in $\cJ_k$ a collection $\{U_i\}$ of open subsets of $M$ is a cover if every subset of $M$ of cardinality $\leq k$ is contained in some $U_i$. 

The $k$th Taylor approximation $T_k(\mr{Emb}(-,N))$ is the homotopy sheafification of $\mr{Emb}(-,N)$ with respect to $\cJ_k$. This means it is up to homotopy the best approximation to $\mr{Emb}(M,N)$ built out of the restrictions of embeddings to $\leq k$ disks in $M$, and hence is explicitly given a right homotopy Kan extension, cf.\ \cite[Definition 4.2]{weisspedrosheaves}: 

\begin{lemma}\label{lem:derived-mapping-space}The $k$th Taylor approximation $T_k(\mr{Emb}(M,N))$ is the derived mapping space 
	\[\bR \mr{Map}_{\cat{PSh}(\cat{Disc}_k)}(\mr{Emb}(-, M),\mr{Emb}(-,N))\]
between the objects $\mr{Emb}(-, M)$ and $\mr{Emb}(-,N)$ of the topological category of space-valued presheaves on the full subcategory $\cat{Disc}_k \subset \cat{Mfd}_n$ on $n$-dimensional manifolds diffeomorphic to a disjoint union of $\leq k$ disks.\end{lemma}

Since every $\cJ_k$-cover is a $\cJ_{k-1}$-cover, we have a Taylor tower as in \cite[Section 3.1]{weisspedrosheaves}
\[\begin{tikzcd} & \vdots \dar\\
\mr{Emb}(M,N) \arrow{ru} \rar \arrow{rd} \arrow{rdd} & T_k(\mr{Emb}(M,N)) \dar \\
& T_{k-1}(\mr{Emb}(M,N)) \dar \\
& \vdots \end{tikzcd}\]
starting at $T_1(\mr{Emb}(M,N))$. By \cite[Proposition 8.3]{weisspedrosheaves}, this tower coincides with the one obtained in \cite[page 84]{weissembeddings}, and hence has the following properties:
\begin{itemize}
	\item Suppose $M$ has handle dimension $h$, then $\mr{Emb}(M,N) \to T_k(\mr{Emb}(M,N))$ is $(-(n-1)+k(n-2-h))$-connected \cite[Corollary 2.6]{goodwillieweiss} (which depends on the results of  \cite{goodwillieklein}).
	%	\item $T_1(E^\mr{id})$ is the space $\mr{Map}_{\half\partial}^{\mr{id}}(M)$ of maps $M \to M$ that are the identity on $\partial M \backslash \mr{int}(D^{n-1})$ and homotopic to the identity.
	\item $T_1(\mr{Emb}(M,N))$ is weakly equivalent to the space $\mr{Imm}(M,N)$ of immersions \cite[page 97]{weissembeddings}. 
	%	\item $T_1(E^\mr{id}) \to T_0(E^\mr{id})$ is inclusion of immersions into continuous maps.
	\item Fix an embedding $\iota \colon M \hookrightarrow N$. For a finite set $I$, let $F_I(M) = \mr{Emb}(I,M)$ be the ordered configuration space. Let $C_k(M) = F_{\{1,\ldots,k\}}(M)/\fS_k$ be the unordered configuration space. Then for $k \geq 2$, the homotopy fiber of $T_k(\mr{Emb}(M,N)) \to T_{k-1}(\mr{Emb}(M,N))$ over the image $\iota$ is weakly equivalent to a relative section space: the space of sections of the bundle over $C_k(M)$ with fiber over a configuration $c \in C_k(M)$ given by $\mr{tohofib}_{I \subset c}(F_I(N))$ that are equal to a section $s^\iota$ near the fat diagonal. We describe the section $s^\iota$ by giving a base point in each $F_I(N)$: recall that $I$ is a collection of points in $M$, then it is the inclusion of $I$ into $N$ by $\iota$. This description appears as \cite[Theorem 9.2]{weissembeddings}.
	%	\item The component of $T_k(E^\mr{id})$ containing the image of the identity is a path-connected $H$-space. This follows from the fact that $T_k F$ is a Taylor approximation. The universal property Taylor approximation implies that the map $E^\mr{id} \times E^\mr{id} \to T_k(\mr{E}^\mr{id})$ factors over $T_k(E^\mr{id} \times E^\mr{id})$, and the description of $T_k$ as a homotopy limit on page 84 of \cite{weissembeddings}  admits a map in from $T_k(E^\mr{id}) \times T_k(E^\mr{id})$.
\end{itemize}

The last point uses total homotopy fibers, whose definition we recall. Let $[1]$ be the category $0 \to 1$, so that a functor $T \colon [1]^k \to \cat{Top}_*$ is a $k$-dimensional cube of pointed spaces. The total homotopy fiber of $T$ is given by
\[\mr{tohofib}(T) = \mr{hofib}[T(0,\ldots,0) \to \mr{holim}_{[1]^k \backslash (0,\ldots,0)} T].\]
It can also be computed by iteratively taking homotopy fibers in each of the $k$ directions. In the non-pointed setting, if $T \colon [1]^k \to \cat{Top}$ has $T(0,\ldots,0)$ path-connected, then $\mr{tohofib}(T)$ is well-defined up to homotopy. 

\subsubsection{Manifolds with boundary}
Now suppose we are interested in manifolds with fixed boundary $Z$. As in \cite[Section 9]{weisspedrosheaves}, we replace $\cat{Mfd}_n$ with the topological category $\cat{Mfd}_{n,Z}$ with objects given by $n$-dimensional smooth manifolds with boundary identified with $Z$ and morphism spaces given by spaces of embeddings that are the identity on $Z$. We can define $\cJ_k$ by letting a collection of open subsets $\{U_i\}$ of $M$ be an cover if every subset in $\mr{int}(M)$ of cardinality $\leq k$ is contained in some $U_i$. If we are given an $n$-dimensional manifold $N$ with an embedding $Z \hookrightarrow \partial N$, we can homotopy sheafify the continuous functor $\mr{Emb}_\partial(-,N) \colon \cat{Mfd}_{n,Z}^\mr{op} \to \cat{Top}$ with respect to these Grothendieck topologies to obtain a Taylor tower. As before, an explicit model for $T_k \mr{Emb}_\partial(M,N)$ is given by $\bR \mr{Map}_{\cat{PSh}(\cat{Disc}_{k,Z})}(\mr{Emb}_\partial(-, M),\mr{Emb}_\partial(-,N))$. Here $\cat{Disc}_{k,Z}$ is the full subcategory of $\cat{Mfd}_{n,Z}$ on $n$-dimensional manifolds that are diffeomorphic rel boundary to a disjoint union of $\leq k$ disks and $Z \times [0,1)$. By \cite[Section 10]{weissembeddings}, this Taylor tower has the same properties as above upon adding ``rel boundary'' where appropriate.

\subsection{Identity component} We will study the identity component using embedding calculus. The following lemma explains how to use towers to prove that a space is in $\cat{\Pi Fin}$. %Note (iii) imposes conditions on the fundamental group of certain path components, the reason for proving Lemma \ref{lem.tkhspaceabelian}.

\begin{lemma}\label{lem.tower} Suppose that a path-connected space $X$ with base point $x$ has the following properties:
\begin{enumerate}[\indent (i)]	
\item There is a tower 
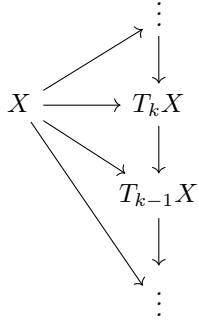
\[\begin{tikzcd} & \vdots \dar\\
	X \arrow{ru} \rar \arrow{rd} \arrow{rdd} & T_k X \dar \\
	& T_{k-1} X \dar \\
	& \vdots \end{tikzcd}\] 
%\[\xymatrix{ & \vdots \ar[d]\\
%X \ar[ru] \ar[r] \ar[rd] \ar[rdd] & T_k X \ar[d] \\
%& T_{k-1} X \ar[d] \\
%& \vdots }\] 
starting at $T_0 X$, such that the map from $X$ to the path component of $T_k X$ containing the image of $x$ is $f(k)$-connected with $\limsup_{k \to  \infty} f(k) = \infty$.
\item The path component of $T_0 X$ containing the image of $x$ has finitely generated homotopy groups.
\item The path component of $T_k X$ containing the image of $x$ has finite or abelian $\pi_1$.
\item The homotopy fiber of $T_k X \to T_{k-1} X$ over the image of $x$ is weakly equivalent to a relative section space $\Gamma_k = \Gamma(E_k,B_k;A_k)$ of a fibration $E_k \to B_k$ with a section, where $B_k$ is a path-connected finite CW-complex, $A_k \subset B_k$ is a non-empty subcomplex and each component of $F_k$ has finitely generated homotopy groups.\end{enumerate}
Then $X \in \cat{\Pi Fin}$.\end{lemma}

\begin{proof}Conditions (i) and (iii) imply that $\pi_1(X)$ will be finite or abelian. Finite groups and finitely generated abelian groups have classifying spaces in $\cat{HFin}$, so it suffices to prove that $\pi_i(X)$ is finitely generated for $i \geq 1$.

Let $x_k \in T_k X$ denote the image of $X$. Using hypothesis (i), if we care about a fixed homotopy group $\pi_i(X,x)$, we can assume that the tower ends at some finite stage $T_K X$ and do an induction over $K$ of the statement that the component of $T_K X$ containing $x_K$ has finitely generated homotopy groups. The initial case $K=0$ is provided by hypothesis (ii).

For the induction step we assume the case $k$ and prove the case $k+1$, using the long exact sequence of homotopy groups for the fiber sequence in hypothesis (iv):
\[\cdots \longto \pi_i(\Gamma_{k+1},x_{k+1}) \longto \pi_i(T_{k+1} X,x_{k+1}) \longto \pi_i(T_{k} X,x_{k}) \longto \cdots, \]
where without loss of generality we may assume $x_{k+1} \in T_{k+1} X$ is the image of $x_{k+1} \in \Gamma_{k+1}$. Using hypothesis (iii), Lemma \ref{lem.sectionspace} says each component of $\Gamma_{k+1}$ has finitely generated homotopy groups. The induction hypothesis says that the path component of $T_{k} X$ containing $x_{k}$ has finitely generated homotopy groups. We then use Lemma \ref{lem.les}(ii), (ii') to finish the induction step (for (ii') use that condition (a) or (b) holds by hypothesis (iii)).\end{proof}

By Lemma \ref{lem.embweqbdy}, $\mr{Emb}_{\half \partial}(M)$ is weakly equivalent to the type of embedding space to which we may apply embedding calculus, and to limit notation we will write $T_k(\mr{Emb}_{\half \partial}(M))$ for the Taylor tower thus obtained. For convergence of the Taylor tower we must verify certain bounds on the dimensions of the handles in a handle decomposition of $M$:

\begin{lemma}\label{lem.handle} Suppose that $n \geq 6$, $M$ is $2$-connected and $\partial M = S^{n-1}$. Then $M$ has a handle decomposition relative to $\partial M \backslash \mr{int}(D^{n-1})$ with only handles of dimension $< n-2$.\end{lemma}

\begin{proof}Pick a disk $D^{n} \subset M$ and write $W = M \backslash \mr{int}(D^{n})$, $V = \partial D^{n} \cong S^{n-1}$, $V' = \partial M \cong S^{n-1}$. It suffices to prove that $W$ admits a Morse function with value $0$ on $V$ and value $1$ on $V'$ with no critical points of index $0$, $1$, $2$, $n-2$, $n-1$ or $n$.

Pick a Morse function $f$ on $W$ with value $0$ on $V$ and value $1$ on $V'$. Remove the critical points of $f$ with index $0$ and $1$ using \cite[Theorem 8.1]{milnorhcobord}. Next remove the critical points of index $2$ using the proof of \cite[Theorem 7.8]{milnorhcobord}, which requires $H_2(W,V) = 0$. This homology group equals $H_2(M,D^{n}) \cong H_2(M) = 0$. A similar argument with $V'$ works to remove the critical points of index $n-2$, $n-1$, or $n$. \end{proof}

\begin{proposition}\label{prop.embidhfin} Suppose that $n \geq 5$ and $M$ has a handle decomposition rel $\partial M \backslash \mr{int}(D^{n-1})$ with handles of dimension $h < n-2$. Then the identity component $\mr{Emb}^\mr{id}_{\half\partial}(M)$ of $\mr{Emb}_{\half\partial}(M)$ is in $\cat{\Pi Fin}$.
\end{proposition}

\begin{proof}We will verify that Lemma \ref{lem.tower} applies to the embedding calculus tower, extended once at the bottom by Smale-Hirsch.\footnote{The reader familiar with embedding calculus may note that our $T_0$ is not the $0$th Taylor approximation. However, we find our notation convenient for this particular argument.}
\[\begin{tikzcd} & \vdots \dar\\
	\mr{Emb}^\mr{id}_{\half\partial}(M) \arrow{ru} \rar \arrow{rd} \arrow{rdd} & T_k(\mr{Emb}_{\half\partial}(M)) \dar \\
	& T_{k-1}(\mr{Emb}_{\half\partial}(M)) \dar \\
	& \vdots\end{tikzcd}\]
%\[\xymatrix{ & \vdots \ar[d]\\
%\mr{Emb}^\mr{id}_{\half\partial}(M) \ar[ru] \ar[r] \ar[rd] \ar[rdd] & T_k(\mr{Emb}_{\half\partial}(M)) \ar[d] \\
%& T_{k-1}(\mr{Emb}_{\half\partial}(M)) \ar[d] \\
%& \vdots }\]

%We check the conditions in Lemma \ref{lem.tower} hold:
\begin{enumerate}[(i)]
	\item If $M$ has handle dimension $h$, the map from $\mr{Emb}^\mr{id}_{\half\partial}(M)$ to the path component of $T_k(\mr{Emb}_{\half\partial}(M))$ containing the image of the identity is $(-(n-1)+k(n-2-h))$-connected. Since $h<n-2$, this goes to $\infty$ as $k \to \infty$.
	\item The identity path component of $T_0(\mr{Emb}_{\half\partial}(M))$ is the space $\mr{Map}_{\half\partial}^{\mr{id}}(M)$ of maps $M \to M$ that are the identity on $\partial M \backslash \mr{int}(D^{n-1})$ and homotopic to the identity rel $\partial M \backslash \mr{int}(D^{n-1})$. Since a mapping space is an example of a section space and simply-connected compact manifolds are in $\cat{Fin}$, by Lemma \ref{lem.sectionspace} $\mr{Map}_{\half\partial}^{\mr{id}}(M)$ has $\pi_i$ for $i \geq 1$ finitely generated.
	\item The identity path component of $T_k(\mr{Emb}_{\half\partial}(M))$ is the identity component of a derived mapping space by Lemma \ref{lem:derived-mapping-space}, which is a path-connected $H$-space and hence has abelian $\pi_1$.	
	\item The cases $k=1$ and $k \geq 2$ are different. We start with the former, and then $T_1(\mr{Emb}_{\half\partial}(M))$ is the space $\mr{Imm}_{\half\partial}(M)$ of immersions $M \looparrowright M$ that are the identity on $\partial M \backslash \mr{int}(D^{n-1})$. The map $T_1(\mr{Emb}_{\half\partial}(M)) \to T_0(\mr{Emb}_{\half\partial}(M))$ is inclusion of immersions into continuous maps. Smale-Hirsh \cite{smaleimm} says that there is a fiber sequence with fiber taken over the identity
	\[\Gamma(\mr{Iso}(TM),M;\partial M \backslash \mr{int}(D^{n-1})) \longto \mr{Imm}_{\half\partial}(M) \longto \mr{Map}_{\half\partial}(M).\]
	
	The fiber given by the space of sections of the bundle over $M$ with fiber over $m \in M$ given by $\mr{Iso}(T_m M) \simeq O(2n)$. By Lemma \ref{lem.sectionspace} the components of the section space $\Gamma(\mr{Iso}(TM),M;\partial M \backslash \mr{int}(D^{n-1}))$ have finitely generated homotopy groups.
		
	Next we discuss the case $k \geq 2$. Recall that for a finite set $I$, $F_I(M) = \mr{Emb}(I,M)$ is the ordered configuration space. There is also an unordered configuration space $C_k(M) = F_{\{1,\ldots,k\}}(M)/\fS_k$. For $k \geq 2$, there is a bundle over $C_k(M)$ with fiber over a configuration $c \in C_k(M)$ given by $\mr{tohofib}_{I \subset c}(F_I(M))$. It has a section $s^\mr{id}$, which can be described by giving compatible base points in the spaces $F_I(M)$; these are given by $\mr{id} \in \mr{Emb}(I,M) = F_I(M)$, after recalling that $I$ is a collection of points in $M$. Then the homotopy fiber of $T_k(\mr{Emb}_{\half\partial}(M)) \to T_{k-1}(\mr{Emb}_{\half\partial}(M))$ is the space of sections of this bundle that are equal $s^\mr{id}$ near the fat diagonal and $\partial M \backslash \mr{int}(D^{n-1})$.
	
	Firstly, we may replace $C_k(M)$ with its homotopy equivalent Fulton-MacPherson compactification $C_k[M]$ \cite{sinha}. This is a finite CW complex and the condition that the sections are equal to $s^\mr{id}$ near the fat diagonal and $\partial M \backslash \mr{int}(D^{n-1})$, becomes that the sections are equal to $s^\mr{id}$ on a certain non-empty subcomplex. Next, we prove that $\mr{tohofib}_{I \subset [k]}(F_I(M))$ is in $\cat{Fin}$. The Fulton-MacPherson compactification $F_I[M]$ of $F_I(M)$ is a finite CW complex and 1-connected, so is in $\cat{Fin}$ by Lemma's \ref{lem.cwhfin} and \ref{lem.finitefin}. By \cite[Theorem B]{goodwillieklein}  $\mr{tohofib}_{I \subset [k]}(F_I(M))$ is $(-(n-3)+k(n-2))$-connected, so $1$-connected as $k \geq 2$. Since a total homotopy fiber is obtained by iterated homotopy fibers and we can disregard the lower homotopy groups, Lemma \ref{lem.les}(iii) suffices. \qedhere

\end{enumerate}\end{proof}

\section{The Weiss fiber sequence} \label{sec.weissfiber} In this section we construct the fiber sequence (\ref{eqn.weissfib}) and its delooping (\ref{eqn.weissfib2}), as discussed in the introduction. There the latter was informally described as
\[ B\mr{Diff}_\partial (M) \longto B\mr{Emb}^{\cong}_{\half\partial}(M) \longto B(B\mr{Diff}_\partial(D^{n}),\natural).\] 
It will arise as the homotopy quotient of an action of a topological monoid on a module:
\[\cat{BM} \longto \cat{BM} \sslash \cat{BD} \longto * \sslash \cat{BD}.\]Subsection \ref{subsec.monoids} discusses the relevant background material on topological monoids and Subsection \ref{subsec.moore} defines $\cat{BD}$ and $\cat{BM}$. In Subsection \ref{subsec.delooping} we then construct the fiber sequence and identify its terms. In Subsection \ref{subsec.gen} we give several useful generalizations.

\begin{remark} There is a choice whether $\mr{Diff}_{\partial}(-)$ consists of diffeomorphisms that (a) are the identity on the boundary, (b) have jet equal to that of the identity on the boundary, or (c) are the identity on a neighborhood of the boundary. The inclusions (c) $\subset$ (b) $\subset$ (a) are weak equivalences, so we will not distinguish between these.\end{remark}

\begin{remark} There is a choice whether to model the classifying space $BM$ for a unital topological monoid $M$ by the thin or thick geometric realization of the nerve $N_\bullet M$. There is always a map $||N_\bullet M|| \to |N_\bullet M|$ which by \cite[Appendix A]{segalcategories} is a weak equivalence if the inclusion of the unit $\{e\} \hookrightarrow M$ is a cofibration. This will always be the case for us, as diffeomorphisms or embeddings are open subsets of infinite-dimensional manifolds \cite{michor}.
\end{remark}

\subsection{Classifying spaces of topological monoids} \label{subsec.monoids} Though we will recall some notation and results below, we assume that the reader is familiar with topological monoids, simplicial spaces, geometric realization, and the double bar construction. These results are well-known, and the expository account \cite{ebertrw} provides proofs and references to the literature.

A \emph{topological monoid} is a unital monoid object in the category of compactly generated weakly Hausdorff spaces,\footnote{Our spaces are always implicitly replaced by their compactly generated weakly Hausdorff replacement, if they are not yet of this type.} i.e.\ a space $\cat{A}$ with maps $m \colon \cat{A} \times \cat{A} \to \cat{A}$ and $u \colon * \to \cat{A}$ satisfying associativity and unit axioms. A left module $\cat{B}$ over a topological monoid $A$ is a left module object over $A$ in the same category, i.e.\ a space $\cat{B}$ with a map $m \colon \cat{A} \times \cat{B} \to \cat{B}$ satisfying associativity and unit axioms. There is a similar definition for right modules.

If $\cat{A}$ is a topological monoid, $\cat{B}$ is a right $\cat{A}$-module and $\cat{C}$ is a left $\cat{A}$-module, the bar construction is the simplicial space $B_\bullet(\cat{B},\cat{A},\cat{C})$ with $p$-simplices given by
\[B_p(\cat{B},\cat{A},\cat{C}) = \cat{B} \times \cat{A}^p \times \cat{C}.\]
The face maps are induced by the monoid multiplication and action maps, and the degeneracy maps induced by the unit of $\cat{A}$. Note that $*$ is always a left and right $\cat{A}$-module, it is in fact both the terminal left $\cat{A}$-module and the terminal right $\cat{A}$-module.

\begin{definition}Let $\cat{A}$ be a topological monoid and $\cat{B}$ be a right $\cat{A}$-module. Then the \emph{homotopy quotient} $\cat{B} \sslash \cat{A}$ is defined to be the thick geometric realization of $B_\bullet(\cat{B},\cat{A},\ast)$. \end{definition}

The thick geometric realization $||-||$ only makes identifications using the face maps, not the degeneracy maps, and is homotopically more well-behaved. In particular, we have \cite[Theorem 2.2]{ebertrw}:

\begin{lemma}\label{lem.geomrelweq} If $f_\bullet \colon X_\bullet \to Y_\bullet$ is a levelwise weak equivalence of simplicial spaces, then $||f_\bullet|| \colon ||X_\bullet|| \to ||Y_\bullet||$ is a weak equivalence.\end{lemma}

%For these constructions to be well-behaved, we need to impose a conditions on the unit of $\cat{A}$. In this paper, a cofibration means a closed Hurewicz cofibration. A topological monoid $\cat{A}$ is \emph{well-pointed} if the inclusion $\{e\} \to \cat{A}$ of the unit is a cofibration. A simplicial space $X_\bullet$ is \emph{good} if each map $s_i(X_q) \to X_{q+1}$ is a cofibration and it is \emph{proper} if the map $\bigcup_i s_i(X_q) \to X_{q+1}$ is a cofibration. Good implies proper \cite{lewiscof}. A useful property of proper simplicial spaces is that if $X_\bullet \to Y_\bullet$ is a simplicial map between proper simplicial spaces that is a levelwise weak equivalence, then $|X_\bullet| \to |Y_\bullet|$ is a weak equivalence, e.g.\ combining \cite[Theorem 2.2]{ebertrw} and \cite[Appendix A]{segalcategories}. The following is a consequence of the elementary fact that a product of a cofibration and an identity map is a cofibration.

%\begin{lemma}Let $\cat{A}$ be a topological monoid, $\cat{B}$ be a right $\cat{A}$-module and $\cat{C}$ be a left $\cat{A}$-module. Then if $\cat{A}$ is well-pointed, $B_\bullet(\cat{B},\cat{A},\cat{C})$ is a good and hence proper simplicial space.\end{lemma}

The following is a consequence of \cite[Theorem 2.12]{ebertrw} (it more generally holds when the topological monoid $\cat{A}$ is \emph{group-like}, i.e.\ if $\pi_0(\cat{A})$ is a group).

\begin{theorem}\label{thm.fibersequencemodule} Let $\cat{A}$ be a topological monoid and $\cat{B}$ be a right $\cat{A}$-module. If $\cat{A}$ is path-connected, then there is a fiber sequence
\[\cat{B} \longto \cat{B} \sslash \cat{A} \longto * \sslash\cat{A}.\]\end{theorem}

An important special case is when $\cat{B} = \cat{A}$. Since $\cat{A}$ has a unit, an extra degeneracy argument as in \cite[Lemma 1.12]{ebertrw} implies that the augmentation $||B_\bullet(\cat{A},\cat{A},*)|| \to *$ is a weak equivalence:

\begin{lemma}\label{lem.aquotienta} If $\cat{A}$ is a path-connected topological monoid, then $\cat{A} \sslash \cat{A} \simeq *$.
\end{lemma}

\begin{corollary}\label{cor.fibersequencemonoid} If $\cat{A}$ is a path-connected topological monoid, then the natural map $\cat{A} \to \Omega(* \sslash \cat{A})$ is a weak equivalence.\end{corollary}

%We want to move loops inside geometric realization, which is possible by Theorem 12.3 of \cite{M}.
%
%\begin{theorem}[May] \label{thm.mayloop} If $X_\bullet$ is a proper pointed simplicial space and each $X_q$ is path-connected, then the natural map $|\Omega X_\bullet| \to \Omega|X_\bullet|$ is a weak equivalence.
%\end{theorem}

We shall need a result about commuting homotopy limits and geometric realization. If $X_\bullet$ is a simplicial space with augmentation to $X_{-1}$, for $x \in  X_{-1}$ we can form the levelwise homotopy fiber $\mr{hofib}_x(X_\bullet \to X_{-1})$ by defining $\mr{hofib}_x(X_p \to X_{-1})$ to be the space of pairs $(y,\gamma) \in X_p \times \mr{Map}([0,1],X_{-1})$ such that $\gamma(0) = \epsilon(y)$ and $\gamma(1) = x$. The following is a consequence of \cite[Lemma 2.14]{ebertrw}:

\begin{lemma}\label{lem.hofibaugmentation} Let $X_\bullet$ be a simplicial space with augmentation to $X_{-1}$, then for each $x \in X_{-1}$, $||\mr{hofib}_x(X_\bullet \to X_{-1})|| \simeq \mr{hofib}_x(||X_\bullet|| \to X_{-1})$.\end{lemma}

%\begin{lemma}\label{lem.hofibproper} If $X_\bullet$ is proper, then $\mr{hofib}_x(X_\bullet \to X_{-1})$ is also proper.
%\end{lemma}
%
%\begin{proof}Write $P_x X_{-1}$ for the space of paths $\gamma$ in $X_{-1}$ ending at $x$. In the explicit model for $\mr{hofib}_x$ given above, we have that $\bigcup_i s_i(\mr{hofib}_x(X_p \to X_{-1})) \cong  \mr{hofib}_x(\bigcup_i s_i(X_p) \to X_{-1})$. In particular we have that $\bigcup_i s_i(\mr{hofib}_x(X_p \to X_{-1}))  \to \mr{hofib}_x(X_{p+1} \to X_{-1})$ is the induced map on the pullbacks of rows in
%\[\begin{tikzcd} P_x X_{-1}  \dar[equals] \rar{\mr{ev}_0} & X_{-1} \dar[equals] & \bigcup_i s_i(X_p) \lar[swap]{\epsilon}  \dar \\ 
%		P_x X_{-1} \rar{\mr{ev}_0} & X_{-1} & X_{p+1}. \lar[swap]{\epsilon} \end{tikzcd} \] 
%As the vertical maps are cofibrations and the left horizontal maps are fibrations, \cite[Theorem 1]{kieboomcof} implies the induced map on pullbacks is a cofibration.
%\end{proof}

\begin{remark}\label{rem.simplicialset} The proofs of some of the results above rely on quasifibrations. However, for topological and PL-manifolds one needs to work in simplicial sets. Since the adjunction $|-| \dashv \mr{Sing}$ is a Quillen equivalence between the category $\cat{sSet}$ of simplicial sets and the category $\cat{Top}$ of CGWH spaces with the Quillen model structures, the geometric realization of the homotopy fiber of a map in $\cat{sSet}$ is weakly equivalent to the homotopy fiber in $\cat{Top}$ of its geometric realization. The same holds for thick geometric realizations by \cite[Lemma 1.7]{ebertrw}. Thus the analogous theorems in $\cat{sSet}$ follow from those in $\cat{Top}$.\end{remark} %We also remark that there is no need for a properness assumption in simplicial sets, see e.g.\ \cite[Proposition IV.1.9]{goerssjardine}. %Theorem \ref{thm.fibersequencemodule} is then Corollary 5.8 of \cite{rezkfib}, while Lemma \ref{lem.hofibaugmentation} is Proposition 5.4 of \cite{rezkfib}. Theorem \ref{thm.mayloop} can be replaced by Corollary IV.4.11 of \cite{goerssjardine}.

\subsection{Moore versions of diffeomorphism groups} \label{subsec.moore} Using a Moore-loop type construction, we will define a topological monoid $\cat{BD}$ which is a strict model for the $H$-space $B\mr{Diff}_{\partial}(D^{n})$ under boundary connected sum $\natural$ and a module $\cat{BM}$ over $\cat{BD}$ which is a strict model for $B\cat{Diff}_\partial(M)$ under the action of $B\mr{Diff}_{\partial}(D^{n})$ by $\natural$.

\subsubsection{Moore monoid of diffeomorphisms of a disk} We start by defining a topological monoid model for $\mr{Diff}_\partial(D^n)$ with boundary connected sum as multiplication. To do this we add a real parameter to constrain the support, and use this parameter to define boundary connected sum by juxtaposition. Hence we will think of our diffeomorphisms as a subspace of the $[0,\infty) \times \mr{Diff}_{\partial}(D^{n-1} \times [0,\infty))$, with the latter having the topology of $C^\infty$-convergence on compacts. Even though $[0,\infty) \times \mr{Diff}_{\partial}(D^{n-1} \times [0,\infty))$ is contractible in this topology, our subspace is not.

\begin{definition}\label{def.mooredisk}
$\cat{D}$ is the \emph{Moore monoid of diffeomorphisms of a disk}, given by the subspace of pairs $(t,\phi) \in [0,\infty) \times \mr{Diff}_{\partial}(D^{n-1} \times [0,\infty))$ such that $\mr{supp}(\phi) \subset D^{n-1} \times [0,t]$.

The multiplication map is given by $\cat{D} \times \cat{D} \ni ((t,\phi),(t',\phi')) \mapsto (t+t',\phi \sqcup \phi') \in \cat{D}$ with $\phi \sqcup \phi' \in  \mr{Diff}_{\partial}(D^{n-1} \times [0,\infty))$ given by
\[(\phi \sqcup \phi')(x,s) \coloneqq \begin{cases}\phi(x,s) & \text{if $s\leq t$,}\\
(\phi'_1(x,s-t),\phi'_2(x,s-t)+t) & \text{otherwise,} \end{cases}\]
and the element $(0,\mr{id})$ is the unit.
\end{definition}

We check this has the desired homotopy type.

\begin{lemma}\label{lem.diffdweq} The inclusion $\mr{Diff}_{\partial}(D^{n-1} \times [0,1]) \hookrightarrow \cat{D}$ given by $\phi \mapsto (1,\phi)$ is a homotopy equivalence.\end{lemma}

\begin{proof}We will homotope $\cat{D}$ onto $\{1\} \times \mr{Diff}_{\partial}(D^{n-1} \times [0,1])$ in two steps. In the first step we decrease the size of the support: the pair $(t,\phi)$ is sent to the path $[0,1] \ni r \mapsto \left(\frac{t}{1+tr},\phi_r\right) \in \cat{D}$ with $\phi_r$ given by
\[\phi_r(x,s) \coloneqq \begin{cases} \left(\phi_1(x,s(1+tr)),\frac{1}{1+tr}\phi_2(x,s(1+tr))\right) & \text{if $s \in [0,\frac{t}{1+tr}]$,} \\
(x,s) & \text{otherwise.} \end{cases}\]
Now that $t < 1$, in the second step we linearly increase $t$ to $1$: 
\[[0,1] \ni r \longmapsto ((1-r)t + r,\phi) \in \cat{D}.\] It is clear that this is homotopic to the identity on $\{1\} \times \mr{Diff}_{\partial}(D^{n-1} \times [0,1])$.
\end{proof}

%\begin{lemma}\label{lem.dgood} $\cat{D}$ is well-pointed.\end{lemma}
%
%\begin{proof}We will give an open neighborhood $U$ of $(0,\mr{id})$ which deformation retracts onto $(0,\mr{id})$. Pick an open convex neighborhood $V$ of $\mr{id}$ in the invertible $(n \times n)$-matrices. The neighborhood $U$ will consist of those pairs $(t,\phi)$ such that $t<1$ and for all $(x,s) \in D^{n-1} \times [0,1]$ the derivative $D_{x,s}(\phi)$ lies in $V$. The deformation retraction of $U$ onto $(0,\mr{id})$ has two steps. The first deforms $(t,\phi)$ onto $(t,\mr{id})$ by linearly interpolating: 
%\[[0,1] \ni r \longmapsto (t,(1-r)\cdot \phi+r\cdot \mr{id}) \in \cat{D}.\]
%The inverse function theorem and our choice of $V$ implies this is a path of diffeomorphisms. The second step deforms the pairs $(t,\mr{id})$ onto $(0,\mr{id})$ by linearly decreasing $t$: 
%\[[0,1] \ni  r \longmapsto ((1-r)t,\mr{id}) \in \cat{D}.\qedhere\]
%\end{proof}

\subsubsection{Moore monoid of classifying spaces of diffeomorphisms of a disk} By \cite[Corollary 11.7]{M}, a unital monoid object in simplicial spaces geometrically realizes to a topological monoid. The same holds for thick geometric realizations using \cite[Remark 2.23]{gepnerhenriques}. Using this and a variation of the nerve construction, we will produce a topological monoid model for $B\mr{Diff}_\partial(D^n)$.

\begin{definition}
The simplicial space $[p] \mapsto \cat{ND}_p$ has $p$-simplices given by $(p+1)$-tuples $(t,\phi_1,\ldots,\phi_p)$ in $[0,\infty) \times \mr{Diff}_{\partial}(D^{n-1} \times [0,\infty))^p$ such that $\bigcup_i \mr{supp}(\phi_i) \subset D^{n-1} \times [0,t]$. The face maps compose diffeomorphisms and the degeneracy maps insert identities. The operations $\sqcup$ described before makes $\cat{ND}_\bullet$ a unital monoid object in simplicial spaces.

We call its thick geometric realization $\cat{BD} \coloneqq ||\cat{ND}_\bullet||$  the \emph{Moore monoid of classifying spaces of diffeomorphisms of a disk}.
\end{definition}

We check this has the desired homotopy type.

\begin{lemma}\label{lem.bdiffbdweq} The inclusion $B\mr{Diff}_{\partial}(D^{n-1} \times [0,1]) \hookrightarrow \cat{BD}$ given by $x \mapsto (1,x)$ is a weak equivalence.\end{lemma}

\begin{proof}There is a map of simplicial spaces $N_\bullet (\mr{Diff}_{\partial}(D^{n-1} \times [0,1])) \to \cat{ND}_\bullet$, which is a levelwise weak equivalence by Lemma \ref{lem.diffdweq}. Using Lemma \ref{lem.geomrelweq} its thick geometric realization is also a weak equivalence.\end{proof}

%\begin{lemma}\label{lem.bdgood} $\cat{BD}$ is well-pointed.\end{lemma}
%
%\begin{proof}\cite[Lemma 11.3]{M} implies that the geometric realization of a well-pointed simplicial space is well-pointed. It thus suffices to prove that each inclusion $\{(0,\mr{id},\ldots,\mr{id})\} \hookrightarrow \cat{D}_k$ is a cofibration. This follows from Lemma \ref{lem.dgood} and the fact that finite products of cofibrations are cofibrations, e.g.\ by \cite[Theorem 1]{kieboomcof}.\end{proof}

\begin{figure}
	\begin{tikzpicture}
	\draw (5,-0.6) -- (2,-0.6) to[out=180,in=0] (0,-1) to[out=180,in=-90] (-1,0) to[out=90,in=180] (0,1) to[out=0,in=180] (2,0.6) -- (5,0.6);
	\draw (.4,0) to[out=-90,in=-90] (-.4,0);
	\draw (.3,-0.15) to[out=90,in=90] (-.3,-0.15);
	\draw [dotted] (2,0.6) to[out=0,in=0] (2,-0.6);
	\draw [densely dotted] (2,0.6) to[out=180,in=180] (2,-0.6);
	\node at (0,-1.3) {$M$};
	\node at (2,-.95) {$\partial M$};
	\node at (4,0) {$\partial M \times [0,\infty)$};
	\node at (2,1.5) {$M_\propto$};
	\draw (1.55,.2) -- (1.8,.2);
	\draw (1.55,-.2) -- (1.8,-.2);
	\node at (1.7,0) [left] {\tiny $D^{n-1}$};
	
	\begin{scope}[xshift=8cm]
	\draw (2,-0.6) to[out=180,in=0] (0,-1) to[out=180,in=-90] (-1,0) to[out=90,in=180] (0,1) to[out=0,in=180] (2,0.6);
	\draw [dotted] (2,-0.6) -- (5,-0.6);
	\draw [dotted] (2,0.6) -- (5,0.6);
	\draw (1.65,-.2) -- (3,-.2) to[out=100,in=-100] (3,.2) -- (1.65,.2);
	\node at (3,0) [right] {$t$};
	\draw (2,-0.6) to[out=180,in=-80] (1.65,-.2);
	\draw (2,0.6) to[out=180,in=80] (1.65,.2);
	\draw [densely dotted] (1.65,-.2) to[out=100,in=-100] (1.65,.2);
	\draw (.4,0) to[out=-90,in=-90] (-.4,0);
	\draw (.3,-0.15) to[out=90,in=90] (-.3,-0.15);
	\draw [dotted] (2,0.6) to[out=0,in=0] (2,-0.6);
	\node at (0,-1.3) {$M$};
	\node at (2,-.95) {$\partial M$};
	\node at (2,1.5) {$M_t$};
	\end{scope}
	\end{tikzpicture}
	\caption{The manifolds $M_t \subset M_\propto$.}
	\label{fig.mtminfty}
\end{figure}

%\begin{figure}[t]\includegraphics[width=\textwidth]{mtminfty}
%	\caption{An example of $M_\propto$ and $M_t$.}
%	\label{fig.mtminfty}
%\end{figure}

\subsubsection{Moore modules of diffeomorphisms and classifying spaces of diffeomorphisms of a manifold} We now generalize the definitions of $\cat{D}$ and $\cat{BD}$ to an $n$-dimensional manifold $M$. To do so, we recall the manifold $M_\propto$, which shall serve as an analogue of $D^{n-1} \times [0,\infty)$ and $D^{n-1} \times [0,t]$:
\[M_\propto = M \cup_{(\partial M \times \{0\})} (\partial M \times [0,\infty)).\]
For $t \in [0,\infty)$ this contains a submanifold with corners:
\[M_t \coloneqq M \cup_{(D^{n-1} \times \{0\})} (D^{n-1} \times [0,t]).\]

\begin{definition}$\cat{M}$ is the \emph{Moore module of diffeomorphisms of $M$}. It is given by the subspace of pairs $(t,\phi) \in [0,\infty) \times \mr{Diff}(M_\propto)$ such that $\mr{supp}(\phi) \subset M_t$.

The right action of $\cat{D}$ on $\cat{M}$ is given by $\cat{M} \times \cat{D} \ni ((t,\phi),(t',\phi')) \mapsto (t+t',\phi \sqcup \phi') \in \cat{M}$ with $\phi \sqcup \phi'$ given by
\[(\phi \sqcup \phi')(m) \coloneqq \begin{cases} \phi(m) & \text{if $m \in M_t$} \\
(\phi'_1(x,s-t),\phi'_2(x,s-t)+t) & \text{if $m = (x,s) \in D^{n-1} \times [t,\infty)$} \\
m & \text{otherwise} \end{cases}\]
\end{definition}

Similarly to above, the thick realization of a module object over a unital monoid object in simplicial spaces is a module over the topological monoid.

\begin{definition}The simplicial space $[p] \mapsto {\cat{N}\cat{M}}_p$ has $p$-simplices given by $(p+1)$-tuples $(t,\phi_1,\ldots,\phi_p)$ in $[0,\infty) \times \mr{Diff}(M_\propto)^p$ such that $\bigcup_i \mr{supp}(\phi_i) \subset M_t$. The operations $\sqcup$ described before makes it a module object in simplicial spaces over the unital monoid object $\cat{ND}_\bullet$.

We call its thick geometric realization $\cat{BM} \coloneqq ||{\cat{N}\cat{M}}_\bullet||$ the \emph{Moore module of classifying spaces of diffeomorphisms of $M$}.\end{definition}

A similar proof as for Lemma's \ref{lem.diffdweq} and \ref{lem.bdiffbdweq} gives the following lemma. 
%For its second part, the first part of the explicit argument in Lemma \ref{lem.dgood} can be replaced by \cite[Proposition 1.2.1]{banyaga}, which says that diffeomorphism groups are locally contractible.

\begin{lemma}\label{lem.bmweq} The maps $\mr{Diff}_\partial(M) \to \cat{M}$ given by $\phi \mapsto (0,\phi)$ and $B\mr{Diff}_\partial(M) \to \cat{BM}$ given by $x \mapsto (0,x)$ are weak equivalences.\end{lemma} %The simplicial space ${\cat{N}\cat{M}}_\bullet$ is proper.

\subsection{A delooped fiber sequence} \label{subsec.delooping} We now construct the fiber sequence (\ref{eqn.weissfib2}):
\[ B\mr{Diff}_\partial (M) \longto B\mr{Emb}^{\cong}_{\half\partial}(M) \longto B(B\mr{Diff}_\partial(D^{n}),\natural).\] 
The precise statement is as follows, and the proof is given in the remainder of this subsection.

\begin{theorem}\label{thm.weissfibersequence} There is a fiber sequence
\[\cat{BM} \longto \cat{BM} \sslash \cat{BD} \longto * \sslash \cat{BD}\]
with $\cat{BM} \simeq B\mr{Diff}_\partial(M)$, $\cat{BM}\sslash\cat{BD} \simeq B\mr{Emb}^{\cong}_{\half\partial}(M)$ and $\Omega (*\sslash\cat{BD}) \simeq B\mr{Diff}_{\partial}(D^n)$. The map $\cat{BM} \to \cat{BM} \sslash \cat{BD}$ is weakly equivalent to the inclusion $B\mr{Diff}_\partial(M) \hookrightarrow B\mr{Emb}^{\cong}_{\half\partial}(M)$.
\end{theorem}

\begin{proof}This follows from Theorem \ref{thm.fibersequencemodule}, because the monoid $\cat{BD}$ is path-connected. The identifications of fiber, total space and base are Lemma \ref{lem.bmweq}, and Propositions \ref{prop.embcat} and \ref{prop.diffdiskloop} respectively. The statement about the map $\cat{BM} \to \cat{BM} \sslash \cat{BD}$ is Lemma \ref{lem.compdiffemb}.\end{proof}

\subsubsection{The base $* \sslash \cat{BD}$} We start by describing the base, by showing that $\ast \sslash \cat{BD}$ is indeed a delooping of $B\mr{Diff}_\partial(D^{n})$.

\begin{proposition}\label{prop.diffdiskloop}We have a weak equivalence
\[B\mr{Diff}_{\partial}(D^{n}) \simeq \Omega (* \sslash \cat{BD}).\]
\end{proposition}

\begin{proof}Since $\cat{BD}$ is path-connected, the map $\cat{BD} \to \Omega (* \sslash \cat{BD})$ is a weak equivalence by Corollary \ref{cor.fibersequencemonoid}. By Lemma \ref{lem.bdiffbdweq} the map $B\mr{Diff}_{\partial}(D^{n-1} \times [0,1]) \to \cat{BD}$ is a weak equivalence. Finally it is standard that $B\mr{Diff}_{\partial}(D^{n-1} \times [0,1]) \simeq B\mr{Diff}_{\partial}(D^{n})$.\end{proof}

\begin{remark}\label{rem.delooping} In fact, $B\mr{Diff}_{\partial}(D^{n})$ is an $n$-fold loop space. This follows from May's recognition principle \cite{M} after remarking it is a path-connected $E_n$-algebra. Smoothing theory provides a particular $n$-fold delooping: $B\mr{Diff}_{\partial}(D^{n}) \simeq \Omega^n_0 \mr{PL}(n)/O(n)$, cf.\ Theorem \ref{thm.smoothingtheory}. One can replace $\mr{PL}(n)$ with $\mr{Top}(n)$ if $n \neq 4$.\end{remark}

\subsubsection{Restrictions on images}
It will be useful to study diffeomorphisms and embeddings with a restriction on the images of certain submanifolds. 
We start with the case of diffeomorphisms.

\begin{definition}Let $\mr{Diff}^{(\mr{im})}(M_\propto)$ be the subspace of $\mr{Diff}(M_\propto)$ consisting of those diffeomorphisms $\varphi$ such that $\varphi(M_0) \subset M_0$.\end{definition}

Using this we can define $\cat{M}^{(\mr{im})}$ as the subspace of $\cat{M}$ consisting of $(t,\varphi) \in [0,\infty) \times \mr{Diff}(M_\propto)$ such that $\mr{supp}(\phi) \subset M_t$ \emph{and}  $\varphi \in \mr{Diff}^{(\mr{im})}(M_\propto)$. The $\cat{D}$-module structure on $\cat{M}$ restricts to a $\cat{D}$-module structure on $\cat{M}^{(\mr{im})}$.

\begin{lemma}\label{lem.diffimageweq} The inclusion $\cat{M}^{(\mr{im})} \hookrightarrow \cat{M}$ is a weak equivalence of $\cat{D}$-modules.\end{lemma}

\begin{proof}Since it is a map of $\cat{D}$-modules, it suffices to prove the map of underlying spaces is a weak equivalence in a manner similar to Lemma \ref{lem.embimageweq}. Suppose we are given a commutative diagram
	\[\begin{tikzcd}S^i \dar \rar & \cat{M}^{(\mr{im})} \dar \\
	D^{i+1} \rar[swap]{f} \arrow[dotted]{ru}& \cat{M},\end{tikzcd}\]
	then we must provide a dotted lift making the diagram commute, possibly after changing it through a homotopy of commutative diagrams.
		
	It suffices to push the image of $M$ under $f_s$ out of $\mr{int}(D^{n-1}) \times (0,\infty)$ for each $s \in D^{i+1}$. To do this, let $\cat{M}'$ be the subspace of $\cat{M}$ consisting of those pairs $(t,\varphi)$ that satisfy the properties that $t \leq 1/2$ and that $\varphi$ is the identity on $(D^{n-1} \backslash \frac{1}{2}D^{n-1}) \times (0,\infty)$. Equivalently, these are the diffeomorphisms $\phi$ supported in $M_0 \cup (\frac{1}{2}D^{n-1} \times [0,1/2])$. In particular $\phi(M_0) \subset M_0 \cup (\frac{1}{2}D^{n-1} \times [0,1/2])$. The inclusion $\cat{M}' \hookrightarrow \cat{M}$ is a weak equivalence by a collar-sliding argument. Hence we may assume that $f_s \in \cat{M}'$.
	
	Now pick a family of compactly-supported diffeomorphisms $\psi_t\colon M_\propto \to M_\propto$ for $t\in [0,1]$ with the following properties:  
	\begin{enumerate}[\indent (i)]
		\item $\psi_t$ is the identity on $(\partial M \backslash \mr{int}(D^{n-1})) \times [0,\infty)$ and $D^{n-1} \times [1,\infty)$, 
		\item $\psi_0 = \mr{id}$,
		%\item $\psi_t(M) \subset \psi_{t'}(M)$ if $t>t'$, and
		\item $\psi_t(M_0 \cup (\frac{1}{2}D^{n-1} \times [0,t])) \subset M_0$.
	\end{enumerate} 
	The diffeomorphism $\psi_t \phi \psi_t^{-1}$ has support $\psi_t(\mr{supp}(\phi))$. Hence conjugating with $\psi_t$ gives us a family $D^{i+1} \times [0,1] \to \cat{M}'$ starting at $f_s$ and ending at diffeomorphisms with support in $M_0$, by condition (iii) for $t=1$. This implies they must send $M_0$ into $M_0$, and hence lie in $\cat{M}^{(\mr{im})}$. By condition (iii) for all $t \in [0,1]$, this homotopy preserves the subspace $\cat{M}^{(\mr{im})}$, so we conclude that $\cat{M}^{(\mr{im})} \hookrightarrow \cat{M}$ is a weak equivalence.
\end{proof}

Recall $\mr{Emb}^{\cong}_{\half\partial}(M'_\propto,M_\propto)$ from Definition \ref{def.embmpm}. Restriction to $M$, resp. $M'_\propto$ gives maps
\[\rho \colon \cat{M}^{(\mr{im})} \longto \mr{Emb}_{\half\partial}(M) \qquad \text{and} \qquad \rho' \colon \cat{M} \longto \mr{Emb}_{\half\partial}(M'_\propto,M_\propto).\]

\begin{lemma}\label{lem.rhocomponents} The maps $\rho$ and $\rho'$ have image in those components of $\mr{Emb}_{\half\partial}(M)$ and $\mr{Emb}_{\half\partial}(M'_\propto,M_\propto)$ consisting of embeddings that are isotopic to a diffeomorphism, resp.\ embedding, which is the identity on $\partial M$ (i.e.\ as indicated by the superscript $\cong$).\end{lemma}

\begin{proof} We use a collaring trick, using the diffeomorphism itself to provide the isotopy. We give the proof in the second case only, the argument in the first case being similar. As in Lemma \ref{lem.diffimageweq}, without loss of generality $(t,\phi) \in \cat{M}'$, i.e.\ $\mr{supp}(\phi) \subset M_0 \cup (\frac{1}{2} D^{n-1} \times [0,1/2])$. Picking a collar $\partial M \times [0,1) \hookrightarrow M$, one may construct a family of diffeomorphisms $\varphi_s \colon M_\propto \to M_\propto$ that are identity on $(\partial M \setminus \mr{int}(D^{n-1})) \times [0,\infty)$, starting at the identity and such that $\varphi_1$ maps $\mr{supp}(\phi)$ into $M$. Restricting the conjugation of $\phi$ with this family to $M'_\propto$ gives the desired isotopy from $\phi|_{M'_\propto}$ to be an embedding that is the identity on $\partial M$.\end{proof}

This describes the vertical maps in a commutative diagram
\[\begin{tikzcd} \cat{M}^{(\mr{im})} \rar{\simeq} \dar & \cat{M} \dar \\
	\mr{Emb}^{\cong}_{\half\partial}(M) \rar{\simeq} & \mr{Emb}^{\cong}_{\half\partial}(M'_\propto,M_\propto).\end{tikzcd}\]
The top map is a weak equivalence by Lemma \ref{lem.diffimageweq}, and the bottom map is a weak equivalence using Lemma's \ref{lem.embimageweq} and \ref{lem.embextweq}. The advantage of the left map is that if both $\cat{M}^{(\mr{im})}$ and $\mr{Emb}^{\cong}_{\half\partial}(M)$ are given the topological monoid structure coming from composition, it is a map of topological monoids. This will be used in the next subsection.

\subsubsection{The total space} We will now finish the proof that $\cat{BM} \sslash \cat{BD} \simeq \smash{B\mr{Emb}^{\cong}_{\half \partial}(M)}$. To do so, we will construct a map $\smash{\cat{BM}^{(\mr{im})} \sslash \cat{BD} \to B\mr{Emb}^{\cong}_{\half \partial}(M)}$. Here $\cat{BM}^{(\mr{im})}$ is constructed as the thick geometric realization of the simplicial subspace ${\cat{N}\cat{M}}^{(\mr{im})}_\bullet$ of ${\cat{N}\cat{M}}_\bullet$ with $p$-simplices given by $(p+1)$-tuples $(t,\phi_1,\ldots,\phi_p)$ in $[0,\infty) \times \mr{Diff}^{(\mr{im})}(M_\propto)^p$. %This is also proper.\marginpar{give proof}

\begin{lemma}The inclusion $\cat{BM}^{(\mr{im})} \hookrightarrow \cat{BM}$ is a weak equivalence.
\end{lemma}

\begin{proof}Using Lemma \ref{lem.geomrelweq}, it suffices to show that ${\cat{N}\cat{M}}^{(\mr{im})}_\bullet \to {\cat{N}\cat{M}}_\bullet$ is a levelwise equivalence. The lemma follows by noting that the inclusion ${\cat{N}\cat{M}}_p^{(\mr{im})} \rightarrow \cat{M}_p$ is weakly equivalent to the $p$-fold product of the inclusion $\cat{M}^{(\mr{im})} \hookrightarrow \cat{M}$, which is a weak equivalence by Lemma \ref{lem.diffimageweq}.\end{proof}

Restriction to $M$ induces a simplicial map ${\cat{N}\cat{M}}^{(\mr{im})}_\bullet \to N_\bullet(\smash{\mr{Emb}^{\cong}_{\half \partial}}(M))$, which geometrically realizes to a map $\cat{BM}^{(\mr{im})} \to B\smash{\mr{Emb}^{\cong}_{\half \partial}}(M)$. As before, $\cat{BD}$-module structure on $\cat{BM}$ restricts to $\cat{BM}^{(\mr{im})}$, and this is a map of $\cat{BD}$-modules when we endow $B\mr{Emb}^{\cong}_{\half \partial}(M)$ with the trivial $\cat{BD}$-module structure. Taking the homotopy quotient by $\cat{BD}$ and projecting, we obtain the map
\[\eta \colon \cat{BM}^{(\mr{im})} \sslash \cat{BD} \longto B\mr{Emb}^{\cong}_{\half \partial}(M) \times * \sslash \cat{BD} \longto  B\mr{Emb}^{\cong}_{\half \partial}(M).\]
We now establish two important properties of this map. The first follows by inspecting the definitions.

\begin{lemma}\label{lem.compdiffemb} The composition 
\[B\mr{Diff}_\partial(M) \hookrightarrow \cat{BM}^{(\mr{im})}  \hookrightarrow \cat{BM}^{(\mr{im})} \sslash \cat{BD} \overset{\eta}{\longrightarrow} B\mr{Emb}^{\cong}_{\half \partial}(M)\] where the first map is as in Lemma \ref{lem.bmweq} and the second is the inclusion of $0$-simplices, coincides with the inclusion $B\mr{Diff}_\partial(M) \hookrightarrow  B\mr{Emb}^{\cong}_{\half \partial}(M)$.\end{lemma}

%\begin{proof}By inspection.\end{proof}

%\begin{figure}[t]\includegraphics[width=\textwidth]{minftyminftyprime}
%\caption{An example of $M_\propto$ and $M'_\propto$.}
%\label{fig.minftyprime}
%\end{figure}

An equivalent description of the map $\eta$ is given by first interchanging the order of thick geometric realization 
\[\cat{BM}^{(\mr{im})} \sslash \cat{BD} \cong ||[p] \longmapsto \cat{NM}_p^{(\mr{im})}  \sslash \cat{ND}_p||,\]
and then taking the thick geometric realization of a simplicial map $\cat{NM}_p^{(\mr{im})}  \sslash \cat{ND}_p \to N_\bullet \mr{Emb}^{\cong}_{\half \partial}(M)$ obtained as follows: we first map $\cat{NM}_p^{(\mr{im})}  \sslash \cat{ND}_p$ to $(\cat{M}^{(\mr{im})}\sslash \cat{D})^p$ by projecting each of the $p$ terms, and then map each of the $p$ terms to $\mr{Emb}_{\half \partial}^{\cong}(M)$ by restricting the element of $\cat{M}^{(\mr{im})}$ to $M$.

\begin{proposition}\label{prop.embcat} The map 
\[\cat{BM}^{(\mr{im})} \sslash \cat{BD} \overset{\eta}{\longto} B\mr{Emb}^{\cong}_{\half\partial}(M)\] is a weak equivalence.\end{proposition}

\begin{proof}On $p$-simplices, the map $\cat{NM}_p^{(\mr{im})}  \sslash \cat{ND}_p \to (\cat{M}^{(\mr{im})}\sslash \cat{D})^p$ is obtained by taking the thick geometric realization of a simplicial map
	\[B_\bullet(\cat{NM}_p^{(\mr{im})},\cat{ND}_p,\ast) \longto \mr{diag}\left(B_{\bullet}(\cat{NM}^{(\mr{im})},\cat{ND},\ast) \times \cdots \times B_{\bullet}(\cat{NM}^{(\mr{im})},\cat{ND},\ast)\right),\]
which is a levelwise weak equivalence by adjusting the support constraints $t$ on the right hand side. By Lemma \ref{lem.geomrelweq} it is a weak equivalence. By \cite[Theorem 7.2]{ebertrw}, the thick geometric realization of the diagonal of $k$-fold simplicial space is weakly equivalent to the iterated thick geometric realization of its $k$ simplicial directions, and we conclude that the map $\cat{NM}_p^{(\mr{im})}  \sslash \cat{ND}_p \to (\cat{M}^{(\mr{im})}\sslash \cat{D})^p$ is a weak equivalence.
	
%Since $[p] \mapsto \cat{NM}_p^{(\mr{im})}  \sslash \cat{ND}_p$ (using \cite[Lemma 11.3]{M}) and $N_\bullet \mr{Emb}^{\cong}_{\half \partial}(M)$ are proper simplicial spaces, 

Again using Lemma \ref{lem.geomrelweq}, it suffices to prove that $\cat{M}^{(\mr{im})} \sslash \cat{D} \to \mr{Emb}^{\cong}_{\half\partial}(M)$ is a weak equivalence. This map fits into a commutative diagram
	\[\begin{tikzcd} \cat{M}^{(\mr{im})} \sslash \cat{D} \rar{\simeq} \dar[swap]{\epsilon} & \dar{\epsilon} \cat{M} \sslash \cat{D} \\
	\mr{Emb}^{\cong}_{\half\partial}(M) \rar{\simeq} & \mr{Emb}^{\cong}_{\half\partial}(M'_\propto,M_\propto),\end{tikzcd}\]
	with top horizontal map a weak equivalence using Lemma \ref{lem.diffimageweq}, bottom horizontal map a weak equivalence using Lemma's \ref{lem.embimageweq} and \ref{lem.embextweq}, and right vertical map a weak equivalence by Lemma \ref{lem.mdtoemb}.\end{proof}

\begin{lemma}\label{lem.mdtoemb} The map $\epsilon \colon \cat{M} \sslash \cat{D} \to \mr{Emb}^{\cong}_{\half\partial}(M'_\propto,M_\propto)$ is a weak equivalence.\end{lemma}

\begin{proof}Consider the submodule $\cat{M}_{\geq 1}$ of $\cat{M}$ consisting of pairs $(t,x)$ with $t \geq 1$. We claim the inclusion $\cat{M}_{\geq 1} \hookrightarrow \cat{M}$ is a weak equivalence. To see this, note that it is in fact a deformation retract, with deformation retraction given by linearly increasing $t$ of an element $(t,\phi) \in \cat{M}$ to a number $\geq 1$:
\[[0,1] \ni r \longmapsto (\max(r,t),\phi) \in \cat{M}.\]
The map induced by restriction to $M'_\propto$ gives an augmentation
\[\epsilon\colon ||B_\bullet(\cat{M}_{\geq 1},\cat{D},*)|| \longto \mr{Emb}^{\cong}_{\half\partial}(M'_\propto,M_\propto).\] 
More precisely, this map is on $p$-simplices $\cat{M}_{\geq 1} \times \cat{D}^p$ given by projecting away the term $\cat{D}^p$ and then applying the map $\cat{M}_{\geq 1} \to \smash{\mr{Emb}^{\cong}_{\half\partial}(M'_\propto,M_\propto)}$ given by restriction to $M_0 \subset M_\propto$.

It suffices to show $\epsilon$ has weakly contractible homotopy fibers. We start by identifying the homotopy fibers using Lemma \ref{lem.hofibaugmentation}: for all $e \in \mr{Emb}^{\cong}_{\half\partial}(M'_\propto,M_\propto)$ we have
\[\begin{tikzcd}{||\mr{hofib}_e(B_\bullet(\cat{M}_{\geq 1},\cat{D},*) \to \mr{Emb}^{\cong}_{\half\partial}(M'_\propto,M_\propto))||} \dar{\simeq} \\ \mr{hofib}_e(||B_\bullet(\cat{M}_{\geq 1},\cat{D},*)|| \to \mr{Emb}^{\cong}_{\half\partial}(M'_\propto,M_\propto)). \end{tikzcd}\] 

Fixing $p \geq 0$, the map $B_p(\cat{M}_{\geq 1},\cat{D},*) = \cat{M}_{\geq 1} \times \cat{D}^p \to \mr{Emb}^{\cong}_{\half\partial}(M'_\propto,M_\propto)$ is given by the composition of the projection map $\cat{M}_{\geq 1} \times \cat{D}^p \to \cat{M}_{\geq 1}$ and the restriction map $\cat{M}_{\geq 1} \to \mr{Emb}^{\cong}_{\half\partial}(M'_\propto,M_\propto)$. Projection is a fibration, so to prove the composite is a fibration it suffices to prove the restriction map is a fibration. This will follow from the parametrized isotopy extension theorem. Suppose we are given for $i \geq 0$ some diagram
\[\begin{tikzcd} D^i \rar{g} \dar & \cat{M}_{\geq 1} \dar{\epsilon} \\
	D^i \times [0,1] \rar{G} & \mr{Emb}^{\cong}_{\half\partial}(M'_\propto,M_\propto).\end{tikzcd}\] 
%\[\xymatrix{D^i \ar[r]^-g \ar[d] & \cat{M}_1 \ar[d]^\epsilon \\
%D^i \times [0,1] \ar[r]_-G & \mr{Emb}^{\cong}_{\half\partial}(M'_\propto,M_\propto)}\] 
Write the top map as $(\tau,\gamma)\colon D^i \to [0,\infty) \times \mr{Diff}(M_\propto)$.  There is a continuous map
\[\Theta\colon \mr{Emb}^{\cong}_{\half\partial}(M'_\propto,M_\propto) \longto [0,\infty),\]
recording for an embedding $e$ the minimal value of $t$ such that embedding $e$ has image contained in $M_t \cup M'_\propto$. We remark that $\tau > \Theta \circ G|_{D^i}$, because we used $\cat{M}_{\geq 1}$ and so the image of $M_0$ is not only contained in $M_{\tau}$ but in fact has an open neighborhood in $M_{\tau}$ (if we had used $\cat{M}$ we only would have $\geq$). Since $D^i$ is compact, we can find  a $\delta>0$ and a continuous function $T\colon D^i \times [0,1] \to [1,\infty)$ such that (a) $T|_{D^i} = \tau$ and (b) $T>\Theta \circ G + \delta$ on $D^i \times [0,1]$.

Recall isotopy extension for embeddings is proven by an argument that in essence amounts to taking the derivative of a family of embeddings, extending this to a time-dependent vector field on the entire manifold and flowing along it. We bring this up, because it implies we can control the support of our isotopies and find a map $\Psi\colon D^n \times [0,1] \to \mr{Diff}_c(M_\propto)$ such that (a) $\Psi(x,0) = \mr{id}$ and $\Psi(x,s) \circ G(x,0) = G(x,s)$ and (b) $\mr{supp}(\Psi(x,s)) \subset M_{\Theta \circ G(x,s) + \delta}$.

Then our lift is given by 
\begin{align*}L\colon D^i \times [0,1] &\longto \cat{M}_{\geq 1} \subset [0,\infty) \times \mr{Diff}(M_\propto) \\
(x,s) &\longmapsto (T(x,s),\Psi(x,s) \circ \gamma(x)).\end{align*}
Condition (a) implies this is a lift, and (b) implies that $\mr{supp}(\Psi(x,s) \circ \gamma(x)) \subset M_{T(x,s)}$.

As a consequence of proving that these maps are fibrations, we can replace the levelwise homotopy fiber $\mr{hofib}_e(B_\bullet(\cat{M}_{\geq 1},\cat{D},*) \to \mr{Emb}^{\cong}_{\half\partial}(M'_\propto,M_\propto))$ with the levelwise fiber $\epsilon^{-1}(e)_\bullet$. This satisfies $||\epsilon^{-1}(e)_\bullet|| \cong ||B_\bullet(\cat{M}_{e,\geq 1},\cat{D},*)||$ with $\cat{M}_{e,\geq 1}$ be the subspace of $\cat{M}_{\geq 1}$ of diffeomorphisms that agree with $e$ on $M_0$.

\vspace{1ex}

Hence it suffices to prove that $||B_\bullet(\cat{M}_{e,\geq 1},\cat{D},*)||$ is weakly contractible. Since homotopy fibers only depend on the path component of the base point, it suffices to check this only for a particular point in each path component. 

Note that the image of any element of $\mr{Diff}_\partial(M)$ in $\mr{Emb}^{\cong}_{\half\partial}(M'_\propto,M_\propto)$ is the identity on $\partial M_0$. By construction $\mr{Diff}_\partial(M) \to \smash{\mr{Emb}^{\cong}_{\half \partial}(M'_\propto,M_\propto)}$ is surjective on path components (this is the reason for including the superscript $\cong$) and thus in each component we can find an embedding $e$ that is equal to the identity on $\partial M_0$, from which it also follows that $\mr{im}(e) = M_0$. In that case, let $\cat{D}_{\geq 1}$ denote the subspace of $\cat{D}$ of pairs $(t,\phi)$ such that $t \geq 1$. Then $||B_\bullet(\cat{M}_{\mr{e},\geq 1},\cat{D},*)||$ is homeomorphic to $||B_\bullet(\cat{D}_{\geq 1},\cat{D},*)||$ and thus weakly equivalent to $||B_\bullet(\cat{D},\cat{D},*)|| = \cat{D} \sslash \cat{D}$, which is weakly contractible by Lemma \ref{lem.aquotienta}. 
\end{proof}

\subsection{Generalizations}\label{subsec.gen} We state three variations to Theorem \ref{thm.weissfibersequence} without proof, as the required modifications are straightforward.

\subsubsection{Identity components} The set of isotopy classes of diffeomorphisms of a path-connected manifold with boundary is acted upon by $\pi_0(\mr{Diff}_\partial(D^n)) \cong \Theta_{n+1}$ by taking the boundary connected sum of representatives and re-identifying the resulting manifold $D^n \natural M$ with $M$. This operation is denoted by $\natural$, and using it we define the second inertia group of $M$ \cite{levine}.

\begin{definition}\label{def.second-inertia} The \emph{second inertia group of $M$ rel $\partial M$} is the subgroup of $\pi_0(\mr{Diff}_\partial(D^n)) \cong \Theta_{n+1}$ of isotopy classes $h$ rel $\partial M$ with the property that $h \natural \mr{id}_M$ is isotopic to $\mr{id}_M$ rel $\partial M$, i.e.\ the stabilizer of $\mr{id}_M$.\end{definition}

\begin{definition}We now define variations on the various spaces defined earlier:\begin{itemize} 
\item Let $\mr{Emb}^\mr{id}_{\half\partial}(M)$ denote the identity component of $\mr{Emb}_{\half\partial}(M)$ and $\mr{Diff}_{\partial}^\mr{id}(M)$ the identity component of $\mr{Diff}_\partial(M)$. 
\item Let $\cat{BM}^\mr{id}$ be the subspace of $\cat{BM}$ where all diffeomorphisms lie in the identity component.
\item If $H$ is a subgroup of $\pi_0( \diff_\partial(D^n))$, let $\mr{Diff}_{\partial}^H(D^n)$ denote the subgroup of $\mr{Diff}_{\partial}(D^n)$ consisting of those connected components.
\item If $H$ is a subgroup of $\pi_0( \diff_\partial(D^n))$, let $\cat{BD}^{H}$ denote the subspace of $\cat{BD}$ where all diffeomorphisms lie in $H$.
\end{itemize}
\end{definition}

The following modification of Theorem \ref{thm.weissfibersequence} is obtained by modifying the definitions and statements appropriately. There are two places where these modifications seemingly impact the argument: Lemma's \ref{lem.rhocomponents} and \ref{lem.mdtoemb}. Lemma \ref{lem.rhocomponents} is now easier; $\rho$ and $\rho'$ clearly map the path-connected spaces $\cat{M}^\mr{im,id}$ and $\cat{M}^\mr{id}$ into the identity components. Lemma \ref{lem.mdtoemb} is the only place the superscript $\cong$ plays a significant role: it is used to show that for each $e$, $||B_\bullet(\cat{M}_{\mr{e},\geq 1},\cat{D},*)||$ is homeomorphic to $||B_\bullet(\cat{D}_{\geq 1},\cat{D},*)||$ by proving that we may assume $e$ is equal to the identity on $\partial M_0$ so that on the complement of $M_0$ a diffeomorphism of $D^{n-1} \times [0,t]$ for some $t \geq 1$ remains. When restricting to $M^\mr{id}$, this complement exactly carries a diffeomorphism of $D^{n-1} \times [0,t]$ whose isotopy class lies in the second inertia group.

\begin{corollary}\label{cor.weissfibersequenceid} Let $H \subset \pi_0(\diff_\partial(D^n))$ be the second inertia group of $M$. There is a fiber sequence
\[\cat{BM}^{\mr{id}} \longto \cat{BM}^{\mr{id}}\sslash\cat{BD}^{H}  \longto  *\sslash\cat{BD}^{H}\]
with $\cat{BM}^\mr{id} \simeq B\mr{Diff}_{\partial}^{\mr{id}} (M)$, $\cat{BM}^{\mr{id}}\sslash\cat{BD}^{H} \simeq B\mr{Emb}^\mr{id}_{\half\partial}(M) $ and $\Omega (*\sslash\cat{BD}^{H}) \simeq B\mr{Diff}_{\partial}^{H}(D^n)$.
\end{corollary}

\subsubsection{Setwise fixed subsets of the boundary} \label{subsubsec.setwise}

Our next generalization concerns diffeomorphisms and embeddings that fix a submanifold $A$ of the boundary setwise, instead of pointwise. Let $M$ be an $n$-dimensional manifold with boundary $\partial M$, $A \subset \partial M$ a codimension zero submanifold and  $D^{n-1} \subset \partial M 
\backslash \mr{int}(A)$ an embedded disk. 

\begin{definition}We now define some variations on the various spaces used before:
\begin{itemize}
\item Let $\mr{Diff}_{\partial,A}(M)$ be the diffeomorphisms that are the identity on $\partial M \backslash A$ and fix $A$ setwise. 
\item Let $\mr{Emb}_{\half\partial,A}(M)$ be the self-embeddings of $M$ that are the identity on $\partial M \backslash (\mr{int}(A) \cup \mr{int}(D^{n-1}))$ and fix $A$ setwise.
\item Let $\mr{Emb}^{\cong}_{\half\partial,A}(M)$ be the self-embeddings of $M$ that are the identity on $ \partial M \backslash (\mr{int}(A) \cup \mr{int}(D^{n-1}))$, fix $A$ setwise, and are isotopic through embeddings satisfying these conditions to a diffeomorphism that is the identity on $\partial M \backslash \mr{int}(A)$ and fixes $A$ setwise.
\end{itemize}\end{definition}

To adapt the proofs of Theorem \ref{thm.weissfibersequence} and Corollary \ref{cor.weissfibersequenceid} to include $A$ as above, no modifications is needed apart from the introduction of $A$'s in the definitions and statements; in none of the proofs the part of $\partial M$ away from $D^{n-1} \subset \partial M$ plays a significant role.

\begin{corollary}\label{cor.weissfibersequencewitha} There is a fiber sequence
\[\cat{BM}_A \longto \cat{BM}_A \sslash \cat{BD} \longto  * \sslash \cat{BD}\]
with $\cat{BM}_A \simeq B\mr{Diff}_{\partial,A}(M)$, $\cat{BM}_A \sslash \cat{BD} \simeq B\mr{Emb}^{\cong}_{\half\partial,A}(M)$ and $\Omega (*\sslash\cat{BD}) \simeq B\mr{Diff}_{\partial}(D^n)$.

Let $H \subset \pi_0 (\diff_\partial(D^n))$ be the second inertia group of $M$ rel $M \setminus A$. There is a fiber sequence
\[\cat{BM}_{A}^{\mr{id}} \longto \cat{BM}_{A}^{\mr{id}}\sslash\cat{BD}^{H}  \longto  *\sslash\cat{BD}^{H}\]
with $\cat{BM}_{A}^{\mr{id}} \simeq B\mr{Diff}_{\partial}^{\mr{id}} (M)$, $\cat{BM}_{A}^{\mr{id}}\sslash\cat{BD}^{H} \simeq B\mr{Emb}^\mr{id}_{\half\partial,A}(M) $ and $\Omega (*\sslash\cat{BD}^{H}) \simeq B\mr{Diff}_{\partial}^{H}(D^n)$.
\end{corollary}

\subsubsection{Homeomorphisms and PL-homeomorphisms}
The final generalization concerns other categories of manifolds: $\mr{CAT} = \topo, \pl$. In this case the Alexander trick tells us that $\mr{CAT}_\partial(D^n) \simeq *$, so the fiber sequences become weak equivalences. Adapting the proof of Corollary \ref{cor.weissfibersequencewitha} to $\mr{CAT}$-manifolds is routine. Firstly, one needs to work in simplicial sets and invoke Remark \ref{rem.simplicialset}. The reason for using simplicial sets is twofold: there is no reasonable topology on $\mr{PL}$-homeomorphisms or $\mr{PL}$-embeddings, and in both $\mr{PL}$ and $\mr{Top}$, all embeddings and families of embeddings should be locally flat for the isotopy extension theorem to be true. Locally flatness for families is not a pointwise condition, so requires the use of simplicial sets. Secondly, one needs to cite the relevant collaring and parametrized isotopy extension theorems, most of which can be found in \cite{siebenmanndef}.

\begin{corollary}\label{cor.weissfibersequenceothercat} There are weak equivalences
\[B\mr{CAT}_{\partial,A}(M) \simeq B\mr{Emb}^\mr{CAT}_{\half\partial,A}(M)\qquad \text{and} \qquad  B\mr{CAT}^\mr{id}_{\partial,A}(M) \simeq B\mr{Emb}^\mr{CAT,id}_{\half\partial,A}(M).\]
\end{corollary}

\section{Proofs of main results}\label{sec.proofs}

In this section we prove the results announced in the introduction, summarized as follows:

\begin{enumerate}[(5.1)]
\item $B\mr{Diff}_\partial(D^{2n})$. This uses the fiber sequence (\ref{eqn.weissfib2}). We understand the total space using embedding calculus through Theorem \ref{thm.emb} and the fiber using the results of Galatius and Randal-Williams.
\item $B\mr{Diff}_\partial(M)$ for $\dim(M) = 2n$. This uses the  fiber sequence (\ref{eqn.weissfib}). We understand the base using embedding calculus and the fiber using the information obtained about $B\mr{Diff}_\partial(D^{2n})$.
\item $B\mr{Diff}_\partial(D^{2n+1})$ and $B\mr{Diff}_\partial(M)$ for $\dim(M) = 2n+1$. These arguments are similar to those in Subsection \ref{subsec.diskeven} and \ref{subsec.mfdeven}. One replaces the results of Galatius and Randal-Williams with those of Botvinnik and Perlmutter, but to apply Theorem \ref{thm.emb} we shall also need some results for even-dimensional manifolds obtained in Subsection \ref{subsec.mfdeven}.
\item $B\mr{Top}(n)$ and $B\mr{PL}(n)$. Using smoothing theory, we relate $\mr{Top}(n)$ and $\mr{PL}(n)$ to diffeomorphisms of disks and we apply the results obtained in Subsections \ref{subsec.diskeven} and  \ref{subsec.diskodd}.
\item $B\mr{Top}_\partial(M)$ and $B\mr{PL}_\partial(M)$. These are studied using smoothing theory and the information obtained about $B\mr{Top}(n)$ and $B\mr{PL}(n)$ in Subsection \ref{subsec.toppl}.
\item $BC(D^n)$, $\mr{Wh}^\mr{Diff}(*)$ and $A(*)$. These follow from Subsections \ref{subsec.diskeven} and \ref{subsec.diskodd}.
\item $B\mr{haut}(M)$ and $B\widetilde{\mr{CAT}}(M)$. These are studied independently of the previous results. We also treat $\widetilde{\mr{CAT}}(M)/\mr{CAT}(M)$.
\end{enumerate}

\subsection{Diffeomorphisms of the even-dimensional disk} \label{subsec.diskeven} Let $\#$ denote connected sum. The manifolds 
\[W_{g,1} \coloneqq (\#_g S^n \times S^n) \backslash \mr{int}(D^{2n})\]
play an important role in the next proof. 

\begin{theorem}\label{thm.evendisk} If $2n \neq 4$, then $B\mr{Diff}_\partial(D^{2n})$ is in $\cat{Fin}$. It is thus in particular of homotopically and homologically finite type.\end{theorem}

\begin{proof}The case $2n=2$ follows from \cite{smaledisk}, so we restrict our attention to $2n \geq 6$. Let us prove that $*\sslash\cat{BD} \in \cat{HFin}$. Consider the fiber sequence of Theorem \ref{thm.weissfibersequence} for $M = W_{g,1}$:
\[\cat{BW}_{g,1} \longto \cat{BW}_{g,1}\sslash\cat{BD} \longto *\sslash\cat{BD}.\] 

To apply Lemma \ref{lem.propfinitetypespaces}(i), it suffices to prove that for fixed $N$, the homology groups $H_*(\cat{BW}_{g,1}\sslash\cat{BD})$ and $H_*(\cat{BW})$ are finitely generated for $* \leq N$ if $g$ is sufficiently large.

On the one hand, in \cite{grwcob,grwstab1} Galatius and Randal-Williams proved that if $2n \geq 6$ and $* \leq \frac{g-3}{2}$ we have an isomorphism $H_*(\cat{BW}) \cong H_{*}(B\mr{Diff}_\partial(W_{g,1})) \cong H_*(\Omega^{\infty}_0MT\theta)$. Here $MT\theta$ is the Thom spectrum of $-\theta^*\gamma$, with $\theta\colon BO(2n)\langle n \rangle \to BO(2n)$ the $n$-connective cover, $\gamma$ the universal vector bundle over $BO(2n)$, and $\Omega^\infty_0$ denotes a component of the infinite loop space. A component of an infinite loop space of a spectrum is of homotopically finite type if the spectrum has finitely generated homotopy groups in positive degrees. A bounded-below spectrum has finitely generated homotopy groups if and only it has finitely generated homology groups. This is true for $MT\theta$ since the Thom isomorphism says that its homology is a shift of the homology of $BO(2n)\langle n \rangle$, which is in $\cat{Fin}$ by Example \ref{exam.bo2nn}.

On the other hand, by Proposition \ref{prop.embcat}, $\cat{BW}_{g,1}\sslash\cat{BD} \simeq B\mr{Emb}^{\cong}_{\half\partial}(W_{g,1})$ and thus by Theorem \ref{thm.emb}, $\cat{BW}_{g,1}\sslash\cat{BD} \in \cat{HFin}$ for all $g$. By Lemma \ref{lem.propfinitetypespaces}(i) we conclude $* \sslash \cat{BD} \in \cat{HFin}$. Since $\cat{BD}$ is path-connected, $*\sslash\cat{BD}$ is simply-connected and using Lemma \ref{lem.finitefin} we see $*\sslash\cat{BD} \in \cat{Fin}$. Proposition \ref{prop.diffdiskloop} says $\Omega(*\sslash\cat{BD}) \simeq B\mr{Diff}_\partial(D^{2n})$, and the latter is in $\cat{\Pi Fin}$. But because we know $\pi_0(\mr{Diff}_\partial(D^{2n}))$ is finite abelian, in fact $B\mr{Diff}_\partial(D^{2n})$ is in $\cat{Fin}$.\end{proof}

Rationally, the cohomological Serre spectral sequence associated to the fiber sequence in the previous proof collapses at $E_2$ in a range. Theorem 1.1 of \cite{grwcob} describes $H^*(\Omega^\infty_0 MT\theta;\bQ)$, so this proposition reduces the problem of computing $\pi_*(B\mr{Diff}_\partial(D^{2n});\bQ)$ to computing $H^*(B\mr{Emb}^{\cong}_{\half \partial}(W_{g,1});\bQ)$.

\begin{proposition}\label{prop.collapse} If $2n \geq 6$, then for $* \leq \frac{g-4}{2}$
\[H^*(\cat{BW}_{g,1}\sslash\cat{BD};\bQ) \cong \bigoplus_{p+q = *} H^p(\Omega^\infty_0 MT\theta;\bQ) \otimes H^q(*\sslash\cat{BD};\bQ).\]\end{proposition}

\begin{proof}By the homological stability results of Galatius and Randal Williams \cite{grwstab1,grwcob}, we may replace $\Omega^\infty_0 MT\theta$ by $B\mr{Diff}_\partial(W_{g,1})$ in the relevant range. Thus the statement is a consequence of Leray-Hirsch, once we prove that $B\mr{Diff}_\partial(W_{g,1}) \to  \cat{BW}_{g,1}\sslash\cat{BD}$ induces a surjection in rational cohomology in the range $* \leq \frac{g-4}{2}$.

By \cite[Lemma 1.4.1]{weissdalian}, $\mr{Emb}^{\cong}_{\half\partial}(W_{g,1})$ is weakly equivalent as a topological monoid to $\mr{Diff}^\ast_{\partial}(W_{g,1})$, the topological group of homeomorphisms that are smooth away from a single point $x \in D^{n-1} \subset \partial W_{g,1}$. Using Lemma \ref{lem.compdiffemb} the map $B\mr{Diff}_\partial(W_{g,1})  \to B\mr{Top}_\partial(W_{g,1})$ can be factored as
\[B\mr{Diff}_\partial(W_{g,1}) \longto \cat{BW}_{g,1}\sslash\cat{BD} \longto B\mr{Top}_\partial(W_{g,1}).\]
In \cite{oscarjohannesblock}, Ebert and Randal-Williams proved that $B\mr{Diff}_\partial(W_{g,1})  \to B\mr{Top}_\partial(W_{g,1})$ is surjective on rational cohomology in the range $* \leq \frac{g-4}{2}$, and hence so is $B\mr{Diff}_\partial(W_{g,1}) \to \cat{BW}_{g,1}\sslash\cat{BD}$.
\end{proof}

%as the free graded-commutative algebra on classes $\kappa_c$ of degree $|c|-2n$, indexed by monomials $c$ in $e, p_{n-1},p_{n-2} \ldots, p_{\lceil\frac{n+1}{4}\rceil}$ (with degrees given by $|e| =2n$, $|p_i| = 4i$) of total degree $> 2n$. As $*\sslash\cat{BD}$ can be delooped further by Remark \ref{rem.delooping} and hence its rational cohomology is graded-commutative with generated by the linear duals of rational homotopy groups, 

\subsection{Diffeomorphisms of $2n$-dimensional manifolds} \label{subsec.mfdeven} We now prove the first corollaries, about diffeomorphisms of 2-connected manifolds. We single out spheres because $\mr{Diff}(S^n) \simeq \mr{Diff}_\partial(D^{n}) \times O(n+1)$, and hence the following is a direct consequence of Theorem \ref{thm.evendisk}.

\begin{corollary}\label{cor.evenspheres}If $2n \neq 4$, then $B\mr{Diff}(S^{2n})$ is of finite type.
\end{corollary}

\begin{corollary}\label{cor.evenmfd}Let $2n \neq 4$. If $M$ is a closed 2-connected oriented smooth $2n$-dimensional manifold and we denote $M^\circ \coloneqq M \backslash \mr{int}(D^{2n})$, then
\begin{enumerate}[\indent (i)]
\item $B\diff_\partial^\mr{id}(M^\circ) \in \cat{Fin}$,
\item $B\diff_\partial(M^\circ) \in \cat{HFin}$,
\item $B\diff^\mr{id}(M) \in \cat{Fin}$,
\item $B\diff^+(M) \in \cat{HFin}$.
\end{enumerate}
\end{corollary}

\begin{proof}\begin{enumerate}[(i)]
		\item We may restrict to $2n \geq 6$ because there is no $2$-connected closed surface. Consider the fiber sequence of Corollary \ref{cor.weissfibersequenceid}, which loops to a fiber sequence 
\begin{equation}\label{eqn.evenmfdfib} B\mr{Diff}_{\partial}^H(D^{2n}) \longto B\mr{Diff}_{\partial}^{\mr{id}}(M^\circ) \longto \cat{BM}^{\circ,\mr{id}}\sslash\cat{BD}^H \end{equation}
with $H$ the second inertia group of $M^\circ$ (cf.\ Definition \ref{def.second-inertia}). Since $\pi_0(\diff_\partial(D^{2n})) \cong \Theta_{n+1}$ is finite, $B\mr{Diff}_{\partial}^H(D^{2n})$ is a finite cover of $B\mr{Diff}_{\partial}(D^{2n})$. Hence by Theorem \ref{thm.evendisk} and Lemma \ref{lem.propfinitetypespaces}, the fiber of (\ref{eqn.evenmfdfib}) is in $\cat{Fin}$. By Proposition \ref{prop.embidhfin}, the base of (\ref{eqn.evenmfdfib}) is in $\cat{Fin}$. By Lemma \ref{lem.les}(ii), (ii') and the fact that base is simply-connected, the total space of (\ref{eqn.evenmfdfib}) is in $\cat{Fin}$.

\item There is a second fiber sequence
\[B\mr{Diff}_{\partial}^{\mr{id}}(M^\circ) \longto B\mr{Diff}_{\partial}(M^\circ) \longto B\pi_0( \mr{Diff}_{\partial}(M^\circ)).\]
By Proposition \ref{prop.embpi0hfin} its base is in $\cat{HFin}$ and by part (i) its fiber is in $\cat{HFin}$. Using Lemma \ref{lem.propfinitetypespaces}(ii) we conclude $B\mr{Diff}_{\partial}(M^\circ) \in \cat{HFin}$.

\item Letting the identity component of $\diff(M)$ act on $\mr{Emb}^+(D^{2n},M)$, we get a fiber sequence
\[\mr{Emb}^+(D^{2n},M) \longto B\diff_{\partial}^{T}(M^\circ) \longto B\diff^\mr{id}(M), \]
where $B\diff_{\partial}^{T}(M^\circ)$ is the classifying space of the subgroup of $\diff_\partial(M^\circ)$ given by those components consisting of diffeomorphisms which become isotopic to the identity when gluing in a disk. Each of these components is homotopy equivalent to $\mr{Diff}^\mr{id}_\partial(M^\circ)$, which by part (i) is in $\cat{Fin}$. The group of components $\pi_0(\diff_{\partial}^{T}(M^\circ))$ is generated by a Dehn twist along the boundary (i.e.\ the diffeomorphisms obtained by inserting an element of $\pi_1(SO(n)) \cong \bZ/2\bZ$ on a collar $[0,1] \times S^{n-1}$ of the boundary of $M^\circ$), and hence is either $\bZ/2\bZ$ or trivial. See \cite[page 302]{wall3} for an example. The fiber is easily seen to lie in $\cat{Fin}$. Thus $B\diff_{\partial}^{T}(M^\circ) \in \cat{Fin}$. By Lemma \ref{lem.les}(i) and the fact that the base is simply-connected, we conclude that $B\diff^\mr{id}(M) \in \cat{Fin}$.

\item We use that \cite[Theorem 13.3]{sullivaninf} proves $B\pi_0(\diff^+(M)) \in \cat{HFin}$, and finish the proof by applying Lemma \ref{lem.propfinitetypespaces}(ii) to the fiber sequence
\[B\diff^\mr{id}(M) \longto B\diff^+(M) \longto B\pi_0(\diff^+(M)).\qedhere\]
\end{enumerate}
\end{proof}

In \cite{oscarjohannestorelli}, Ebert and Randal-Williams studied the higher-dimensional analogue of the Torelli group, $\mr{Tor}_\partial(W_{g,1})$, defined as the subgroup of $\mr{Diff}_\partial(W_{g,1})$ given by components that act trivially on $H_n(W_{g,1})$. For surfaces of genus $\geq 7$, some rational homology group of the Torelli group is infinite-dimensional \cite{akita}. In high dimensions this is not the case:

\begin{corollary}For $2n \geq 6$, $B\mr{Tor}_\partial(W_{g,1})$ is of homologically finite type.\end{corollary}

\begin{proof}By work of Kreck \cite{kreckisotopy}, the components of $\mr{Tor}_\partial(W_{g,1})$ are an extension of a finitely generated abelian group by a finitely generated abelian group. Such groups have classifying spaces in $\cat{HFin}$ by Lemma's \ref{lem.emhfin} and \ref{lem.groupsprops}. The corollary then follows from Corollary \ref{cor.evenmfd}(i), Lemma \ref{lem.propfinitetypespaces}(ii), and the fiber sequence
\[B\mr{Diff}_{\partial}^\mr{id}(W_{g,1}) \longto B\mr{Tor}_\partial(W_{g,1}) \longto B\pi_0 (\mr{Tor}_\partial(W_{g,1})).\qedhere\]
\end{proof}

\subsection{Diffeomorphisms of the odd-dimensional disk} \label{subsec.diskodd} For the odd-dimensional case, we replace $W_{g,1}$ with the manifold
\[H_g \coloneqq \natural_g (D^{n+1} \times S^n),\]
where $\natural$ denotes boundary connected sum. Its boundary is $\partial H_g \cong W_g$. We establish some notation. \begin{itemize}
\item Fix a disk $D^{2n} \subset \partial H_g$, so that $\partial H_g \backslash \mr{int}(D^{2n}) \cong W_{g,1}$. Recall from Subsection \ref{subsubsec.setwise} that $\mr{Diff}_{\partial,W_{g,1}}(H_g)$ is the subgroup of $\mr{Diff}(H_g)$ consisting of those diffeomorphisms that fix $D^{2n}$ pointwise and $\partial H_g \backslash \mr{int}(D^{2n})$ setwise.
\item  Fix a smaller disk $D' \cong D^{2n} \subset \mr{int}(D^{2n})$. Recall that $\smash{\mr{Emb}^{\cong}_{\half\partial,W_{g,1}}(H_g)}$ is the subspace of $\mr{Emb}(H_g)$ consisting of those self-embeddings that fix $D^{2n} \backslash \mr{int}(D')$ pointwise and fix $\partial H_g \backslash \mr{int}(D^{2n})$ setwise, and are isotopic through such embeddings to a diffeomorphism.
\item Let $\mr{Diff}^\mr{ext}_{\partial}(W_{g,1})$ be the subgroup of $\mr{Diff}_\partial(W_{g,1})$ consisting of those diffeomorphisms of $W_{g,1}$ fixing the boundary that extend to a diffeomorphism of $H_g$ that fixes $D^{2n}$.
\end{itemize} 

Proposition \ref{prop.embpi0hfin} does not apply to the connected components of the last of these groups. Instead we will use work of Kreck and Wall on isotopy classes of diffeomorphisms of handlebodies \cite{kreckisotopy,wall3}. Wall's results are stated for pseudoisotopy classes, but pseudoisotopy classes coincide with isotopy classes for simply-connected manifolds of dimension $\geq 5$ \cite{cerf}.

\begin{lemma}\label{lem.hgpi0} If $n \geq 4$, the group $\pi_0( \mr{Diff}^\mr{ext}_{\partial}(W_{g,1}))$ has a classifying space in $\cat{HFin}$.\end{lemma}

\begin{proof}Kreck gave a complete description of $\pi_0(\mr{Diff}_\partial(W_{g,1}))$ up to extensions in terms of two short exact sequences \cite[Section 7]{sorenoscarabelian}. We will only need one:
\[1 \longto I^n_{g,1} \longto \pi_0(\mr{Diff}_\partial(W_{g,1})) \longto \mr{Aut}(H_n,\lambda,\alpha) \longto 1.\]
Here $H_n$ is the middle-dimensional homology group $H_n(W_{g,1})$ and $\mr{Aut}(H_n,\lambda,\alpha)$ is the arithmetic group of automorphisms of the intersection form and its quadratic refinement. The group $I^n_{g,1} $ is an extension of finitely generated abelian group by a finitely generated abelian group. Thus $\pi_0(\mr{Diff}_\partial(W_{g,1}))$ is an iterated extension of an arithmetic group by finitely generated abelian groups, confirming the first part of Proposition \ref{prop.embpi0hfin} in this particular case.

Theorem 6 of \cite{wall3} says we can describe the subgroup of isotopy classes of diffeomorphisms that extend to $H_g$ by replacing $\mr{Aut}(H_n,\lambda,\alpha)$ with the subgroup consisting of those elements that preserve $K = \ker(H_n(W_{g,1}) \to H_n(H_g)) \subset H_n$, and replacing $I^n_{g,1}$ by a subgroup. This is also an iterated extension of an arithmetic group by finitely generated abelian groups, so has a classifying space in $\cat{HFin}$.\end{proof}

\begin{corollary}\label{cor.odddisk} If $2n+1 \neq 5,7$, then $B\mr{Diff}_\partial(D^{2n+1})$ is in $\cat{Fin}$. It is thus in particular of homologically and homotopically finite type.
\end{corollary}

\begin{proof}The cases $2n+1=1,3$ are respectively folklore and \cite{hatcherdisk}, so we focus on the case $2n+1 \geq 9$. Similar to the proof in Theorem \ref{thm.evendisk}, we start by remarking it suffices to prove that $*\sslash\cat{BD} \in \cat{Fin}$ using the fiber sequence of Corollary \ref{cor.weissfibersequencewitha} applied to $M = H_g$ and $A = W_{g,1} \subset \partial H_g$. Its middle term $\cat{BH}_{g,W_{g,1}} \sslash \cat{BD}$ is weakly equivalent to $B\mr{Emb}^{\cong}_{\half\partial,W_{g,1}}(H_g)$ and there is a fiber sequence
\[B\mr{Emb}^{\cong}_{\half\partial}(H_g) \longto B\mr{Emb}^{\cong}_{\half\partial,W_{g,1}}(H_g) \longto B\mr{Diff}^\mr{ext}_{\partial}(W_{g,1}).\]

The base has a classifying space in $\cat{\Pi Fin}$, using Corollary \ref{cor.evenmfd} and Lemma \ref{lem.hgpi0} for the identity component and group of components respectively, and hence in $\cat{HFin}$ by Lemma \ref{lem.pfinhfin}. The fiber has a classifying space in $\cat{HFin}$ using Theorem \ref{thm.emb}.
Hence using Lemma \ref{lem.propfinitetypespaces}(ii) we obtain that $\cat{BH}_{g,W_{g,1}} \sslash B\cat{D} \in \cat{HFin}$. From this point, the argument continues as in the proof of Theorem \ref{thm.evendisk} with the results of Botvinnik-Perlmutter \cite{perlmutterbotvinnik,perlmutterstab} replacing those of Galatius-Randal-Williams \cite{grwcob,grwstab1}. These results say that for $2n+1 \geq 9$ and $\ast \leq \frac{g-4}{2}$, we have an isomorphism $H_*(B\mr{Diff}_{D^{2n}}(H_g)) \cong H_*(Q_0 BO(2n + 1)\langle n \rangle_+)$. Here $\mr{Diff}_{D^{2n}}(H_g)$ denotes the topological group of diffeomorphisms of $H_g$ fixing $D^{2n} \subset \partial H_g$ pointwise (and necessarily $\partial H_g$ setwise).
\end{proof}

\begin{remark}The restriction $2n+1 \neq 5,7$ is due to the use of a higher version of the Whitney trick. Just like the ordinary Whitney trick is limited to dimension $\geq 5$ due to an application of transversality to make an immersed Whitney disk embedded, a transversality result limits the higher version of the Whitney trick to dimensions $\geq 9$. It may be that automorphisms of manifolds of dimension 5 and 7 behave in a qualitatively different manner.\end{remark}

A similar proof as that for Corollaries \ref{cor.evenspheres} and \ref{cor.evenmfd} now gives us the odd-dimensional case.

\begin{corollary}\label{cor.oddspheres}Let $2n+1 \neq 5,7$, then $B\mr{Diff}(S^{2n+1})$ is of finite type.
\end{corollary}

\begin{corollary}\label{cor.oddmfd} If $M$ is a closed compact 2-connected $(2n+1)$-dimensional oriented smooth manifold and $2n+1 \neq 5,7$ and we denote $M^\circ \coloneqq M \backslash \mr{int}(D^{2n+1})$, then
\begin{enumerate}[\indent (i)]
\item $B\diff_\partial^\mr{id}(M^\circ) \in \cat{Fin}$,
\item $B\diff_\partial(M^\circ) \in \cat{HFin}$,
\item $B\diff^\mr{id}(M) \in \cat{Fin}$,
\item $B\diff^+(M) \in \cat{HFin}$.
\end{enumerate}
\end{corollary}

\subsection{Homeomorphisms and PL-homeomorphisms of $\bR^n$} \label{subsec.toppl}

Recall that $\mr{Top}(n)$ denotes the topological group of homeomorphisms of $\bR^n$ in the compact-open topology and $\mr{PL}(n)$ denote the simplicial group of piecewise-linear homeomorphisms of $\bR^n$. We recall two results about these groups, the first from \cite[Section V.5]{kirbysiebenmann}:
\begin{theorem}[Kirby-Siebenmann] \label{thm.ks} If $n \geq 5$, then $\pi_i(\mr{PL}(n)/O(n)) \to \pi_i(\mr{PL}/O)$ is an isomorphism for $i \leq n+1$ and a surjection for $i=n+2$. Furthermore, there are isomorphisms \[\pi_i(\mr{PL}/O)  \cong \begin{cases} 0 & \text{if $ i \leq 4$,} \\ \Theta_i & \text{otherwise,}\end{cases} \quad \pi_i(\mr{Top}(n)/\mr{PL}(n)) \cong \pi_i(\mr{Top}/\mr{PL}) \cong \begin{cases} \bZ/2\bZ & \text{if $i=3$,} \\ 0 & \text{otherwise.}\end{cases}\]
\end{theorem}

The second result is smoothing theory for PL-manifolds, e.g.\ \cite[Theorem I.4.4]{burglashof}.

\begin{theorem}[Burghelea-Lashof] \label{thm.smoothingtheory} $\Theta_n \times B\mr{Diff}_{\partial}(D^{n}) \simeq \Omega^n \mr{PL}(n)/O(n)$.
\end{theorem}

\begin{corollary}\label{cor.toppl} Let $n \neq 4,5,7$. $B\mr{Top}(n)$ and $B\mr{PL}(n)$ are in $\cat{Fin}$. They are thus in particular of homotopically and homologically finite type.\end{corollary}

\begin{proof}For $n \leq 3$ it is true that $O(n) \simeq \pl(n) \simeq \topo(n)$, so we will focus on $n \geq 6$. Since $O(n)$ is homotopically finite type and $\Theta_i$ is a finite abelian group, Theorem \ref{thm.ks} implies that $\pi_i(O(n)/ \pl(n))$ is finite for $i \leq n+1$. Theorem \ref{thm.smoothingtheory}  says $\Theta_n \times B\mr{Diff}_{\partial}(D^{n}) \simeq \Omega^n \mr{PL}(n)/O(n)$. Since $B\mr{Diff}_{\partial}(D^{n}) \in \cat{Fin}$ by Theorem \ref{thm.main}, $\Omega^{n} \mr{PL}(n)/O(n) \in \cat{Fin}$ and thus we also have $\mr{PL}(n)/O(n) \in \cat{Fin}$. There is a fiber sequence
\[\mr{PL}(n)/O(n) \longto BO(n) \longto B\mr{PL}(n), \]
and since $BO(n) \in \cat{Fin}$, the long exact sequence for homotopy groups implies $B\mr{PL}(n) \in \cat{Fin}$. That $B\topo(n) \in \cat{Fin}$ then follows from the last part of Theorem \ref{thm.ks}.
\end{proof}

\subsection{Homeomorphisms of $2n$-dimensional manifolds} Next, we prove similar results for homeomorphisms and PL-homeomorphisms.

\begin{corollary}\label{cor.topspheres}Let $n\neq 4,5,7$, then $B\mr{Top}(S^n)$ and $B\mr{PL}(S^n)$ are of finite type.
\end{corollary}

\begin{proof}Let $\mr{CAT} = \mr{Top}$ or $\mr{PL}$. Let $\mr{Fr}^\mr{CAT}(TS^n)$ be the bundle over $S^n$ with fiber over $x \in S^n$ given by the $\mr{CAT}$-isomorphisms $\bR^n \to T_x S^n$. A homeomorphism of $S^N$ acts on the tangent microbundle $TS^n$,  giving rise to a fiber sequence
\[B\mr{CAT}_\partial(D^n) \longto B\mr{CAT}(S^n) \longto \mr{Fr}^\mr{CAT}(TS^n).\]
By the Alexander trick $\mr{CAT}_\partial(D^n) \simeq *$, so $B\mr{CAT}(S^n) \simeq \mr{Fr}^\mr{CAT}(TS^n)$. The fiber of $\mr{Fr}^\mr{CAT}(TS^n)$ is equivalent to $\mr{CAT}(n)$, so we can prove the result using Corollary \ref{cor.toppl} by applying Lemma \ref{lem.les}(ii), (ii') to the fiber sequence
\[\mr{CAT}(n) \longto \mr{Fr}^\mr{CAT}(TS^n) \longto S^n. \qedhere\]
\end{proof}

For general manifolds, the technique is smoothing theory for embeddings \cite{lashofembeddings}.

\begin{corollary}\label{cor.mfdcat} Suppose that $M$ is a closed compact 2-connected $n$-dimensional smooth manifold, and $n \neq 4,5,7$, and let $M^\circ \coloneqq M \backslash \mr{int}(D^n)$. If $\mr{CAT} = \topo,\pl$ then $B\mr{CAT}_\partial(M^\circ)$ and $B\mr{CAT}(M)$ are of homologically finite type. Furthermore, the classifying spaces of their identity components are of homotopically and homologically finite type.
\end{corollary}

\begin{proof}We start with the identity component of $\mr{Diff}_\partial(M^\circ)$. Let $n \geq 5$ and $\mr{CAT} = \pl,\topo$, then  Corollary 2 of \cite{lashofembeddings} says that
\[\mr{hofib}_{\mr{id}}(\mr{Imm}^{\cong}_{\half\partial}(M^\circ) \to \mr{Imm}^\mr{CAT}_{\half\partial}(M^\circ)) \simeq \mr{hofib}_\mr{id}(\mr{Emb}^{\cong}_{\half\partial}(M^\circ) \to \mr{Emb}^\mr{CAT}_{\half\partial}(M^\circ))\]
By immersion theory \cite{haefligerimm,leesimm}, the left hand side is equivalent to the homotopy fiber $F$ of a map $\Gamma(E^O,M^\circ;\half \partial M^\circ) \to  \Gamma(E^\mr{CAT},M^\circ;\half \partial M^\circ)$ between spaces of sections of bundles $E^O$ and $E^\mr{CAT}$ with fibers $O(n)$ and $\mr{CAT}(n)$ respectively, equal to a fixed section on half the boundary. As a consequence of Lemma \ref{lem.sectionspace} and Corollary \ref{cor.toppl}, each of the components of $F$ is in $\cat{Fin}$. Furthermore, since the first $n$ homotopy groups of $ \mr{CAT}(n)/O(n)$ are finite by Theorem \ref{thm.ks}, obstruction theory tells us $F$ has finitely many components. 

We conclude there is a fiber sequence
\[F \longto \mr{Emb}^{\diff,\mr{id}}_{\half\partial}(M^\circ) \longto \mr{Emb}^{\mr{CAT},\mr{id}}_{\half\partial}(M^\circ)\]
with $F$ in $\cat{Fin}$. By Proposition \ref{prop.embidhfin}, $\mr{Emb}^{\diff,\mr{id}}_{\half\partial}(M^\circ) \in \cat{\Pi Fin}$. We can then use Lemma \ref{lem.les}(i), (i') to conclude $\smash{\mr{Emb}^{\mr{CAT},\mr{id}}_{\half\partial}(M^\circ)} \in \cat{\Pi Fin}$. By Corollary \ref{cor.weissfibersequenceothercat}, $B\mr{CAT}^{\mr{id}}_{\half\partial}(M^\circ) \simeq B\mr{Emb}_{\half\partial}^{\mr{id}}(M^\circ)$, so the former is in $\cat{Fin}$. 

Next is the group of connected components of $\mr{Diff}_\partial(M^\circ)$. The proof of \cite[Theorem 13.3]{sullivaninf} only uses surgery theory and that the first $n$ rational homotopy groups of $O(n)$ are finite-dimensional, so it also applies to $\mr{CAT}$-manifolds (cf.\ \cite[page 3392]{trianta}). As in Proposition \ref{prop.embpi0hfin}, $B\pi_0( \mr{CAT}_{\half\partial}(M^\circ)) \in \cat{HFin}$. Applying Lemma \ref{lem.propfinitetypespaces}(ii) to
\[B\mr{CAT}_{\half\partial}^{\mr{id}}(M^\circ) \longto B\mr{CAT}_{\half\partial}(M^\circ) \longto B\pi_0( \mr{CAT}_{\half\partial}(M^\circ))\]
we conclude that $B\mr{CAT}_{\half\partial}(M^\circ) \in \cat{HFin}$.

As in the smooth case, there is a fiber sequence
\[\mr{Emb}^{\mr{CAT},+}(D^n,M) \longto B\mr{CAT}_{\partial}^{T}(M^\circ) \longto B\mr{CAT}^\mr{id}(M), \]
where $\mr{Emb}^{\mr{CAT},+}(D^n,M)$ is the space of orientation-preserving embeddings and $B\mr{CAT}_{\partial}^{T}(M^\circ)$ is the classifying space of the subgroup of $\mr{CAT}_\partial(M^\circ)$ given by those components consisting of $\mr{CAT}$-isomorphisms which become isotopic to the identity after gluing in a disk. As before, this group of components is either $\bZ/2\bZ$ or trivial, and we conclude that $B\mr{CAT}^\mr{id}(M) \in \cat{\Pi Fin}$. We also have $B\pi_0( \mr{CAT}(M)) \in \cat{HFin}$, which finishes the proof.
\end{proof}

\subsection{Concordance diffeomorphisms and the $\mr{Wh}^{\diff}$-spectrum} \label{subsec.concordance} Finally, we give an application to concordance diffeomorphisms and algebraic K-theory of spaces. 

\begin{definition}\label{def.concordance}
Let $C(M)$ be the topological group of diffeomorphisms of $M \times I$ fixing $\partial M \times I$ and $M \times \{0\}$ pointwise. These are called the \emph{concordance diffeomorphisms}.\end{definition}

\begin{corollary}\label{cor.concordance} If $n \geq 8$, then $C(D^n)$ is of finite type.
\end{corollary}

\begin{proof}There is a fiber sequence $\mr{Diff}_\partial(D^{n+1}) \to C(D^n) \to \mr{Diff}_\partial(D^n)$  so the result follows from Theorem \ref{thm.evendisk} and Corollary \ref{cor.odddisk}.
\end{proof}

Igusa's pseudoisotopy stability theorem \cite{igusastab} says there is a stabilization map $C(M \times I^k) \to C(M \times I^{k+1})$, which is an equivalence in a range going to $\infty$ as $k \to \infty$. Assume for simplicity that $M$ is 1-connected. By Waldhausen's parametrized stable $h$-cobordism theorem, the space $\mr{colim}_{k \to \infty} BC(M)$ is the infinite loop space $\Omega^{\infty+1} \mr{Wh}^{\diff}(M)$ \cite{waldhausenstable}. The spectrum $\mr{Wh}^\mr{Diff}(M)$ is a summand of $A(M)$: $A(M) \simeq \mr{Wh}^\mr{Diff}(M) \vee \Sigma^\infty M_+$. Both $A(M)$ and $\mr{Wh}^\mr{Diff}(M)$ only depend on the homotopy type of $M$. Since $A(M)$ is connective and $\pi_0(A(M)) \cong \bZ$, the splitting implies that $\pi_i(\mr{Wh}^{\diff}(M)) = 0$ for $i \leq 0$. Thus Corollary \ref{cor.concordance} implies the following result of Dwyer \cite{dwyertwisted}:
 
\begin{corollary}\label{cor.whdiff} Both $\mr{Wh}^\mr{Diff}(*)$ and $A(*)$ are of finite type.\end{corollary}

\subsection{Homotopy automorphisms and block automorphisms} Two types of automorphism groups of manifolds remain to be discussed: homotopy automorphisms and block automorphisms. Finiteness results for these can be obtained with relative ease, and we include their proofs for completeness. Similar results were obtained by Farrell-Hsiang and Burghelea \cite{farrellhsiang,burgheleafiniteness}, and more recently, rationally by Berglund and Madsen \cite{berglundmadsen2}.  We will assume that our manifolds are closed, but we expect this restriction to be unnecessary.

\begin{definition}Let $\mr{haut}(M)$ be the group-like topological monoid of \emph{homotopy automorphisms} of $M$, i.e.\ the union of those components of $\mr{Map}(M,M)$ consisting of continuous maps that are homotopy equivalences.\end{definition}

\begin{proposition}\label{prop.haut} If $M$ is a closed $n$-dimensional manifold with finite fundamental group, then
\begin{enumerate}[\indent i)]
\item $B\mr{haut}^\mr{id}(M)  \in \cat{Fin}$,
\item $B\mr{haut}(M) \in \cat{HFin}$.
\end{enumerate}\end{proposition}

\begin{proof}\begin{enumerate}[(i)]
		\item Since the identity component of $\mr{haut}(M)$ is a path-connected topological monoid, it suffices to prove it has finitely-generated homotopy groups. By Lemma \ref{lem.sectionspace}, the identity component of the space $\mr{Map}_*(M,M)$ of pointed maps has finitely-generated homotopy groups. It fits into a fiber sequence
\[\mr{Map}^\mr{id}_*(M,M) \longto \mr{Map}^\mr{id}(M,M) \longto M,\]
and the result follows from Lemma \ref{lem.les}(ii), (ii') using the fact that there is a section $M \to \mr{Map}^\mr{id}(M,M)$.

	\item After having established part (i), we want to apply Lemma \ref{lem.propfinitetypespaces}(ii) to the fiber sequence
\[B\mr{haut}^\mr{id}(M) \longto B\mr{haut}(M) \longto B\pi_0(\mr{haut}(M)),\]
and to do so, it suffices to show $B\pi_0(\mr{haut}(M)) \in \cat{HFin}$. Triantafillou proved that for a finite CW complex $X$ with finite fundamental group $\pi_0(\mr{haut}(X))$ is arithmetic \cite{trianta}, and hence we can apply Theorem \ref{thm.arithmetic}.\qedhere
\end{enumerate}
\end{proof}

Next we consider block automorphisms. We take $\mr{CAT} = \mr{Top}$, $\mr{PL}$ or $\mr{Diff}$.

\begin{definition}The group $\widetilde{\mr{CAT}}(M)$ of \emph{$\mr{CAT}$ block automorphisms}  of a $\mr{CAT}$-manifold $M$ is the simplicial group with $k$-simplices given by the $\mr{CAT}$-isomorphisms 
\[f\colon \Delta^k \times M \longto \Delta^k \times M\]
such that $f(\sigma \times M) = \sigma \times M$ for every face $\sigma$ of $\Delta^k$.\end{definition}

This group is designed to be studied by surgery. The difference between block automorphisms and homotopy automorphisms is the subject of Quinn's surgery exact sequence \cite{quinnsurgery}: if $n \geq 5$ (or $n = 4$, $\mr{CAT} = \mr{Top}$ and $\pi_1(M)$ is good \cite{freedmanquinn}) the following is a fiber sequence when restricted to identity components
\begin{equation}\label{eqn:surgery} \mr{haut}(M)/\widetilde{\mr{CAT}}(M) \longto \mr{Map}(M,G/\mr{CAT}) \longto \bL(M),\end{equation}
where $G = \mr{colim}_{n \to \infty} \mr{haut}_*(S^n)$ and $\bL(-)$ is the quadratic $L$-theory space.

\begin{proposition}\label{prop.block} If $M$ is a closed oriented $\mr{CAT}$-manifold of dimension $n \geq 5$ with finite fundamental group or a 1-connected topological 4-manifold, then
\begin{enumerate}[\indent (i)]
\item $B\widetilde{\mr{CAT}}^\mr{id}(M) \in \cat{Fin}$,
\item $B\widetilde{\mr{CAT}}(M) \in \cat{HFin}$.
\end{enumerate}\end{proposition}

\begin{proof}\begin{enumerate}[(i)] \item There is a fiber sequence
\[\mr{haut}^\mr{id}(M)/\widetilde{\mr{CAT}}^\mr{id}(M) \longto B\widetilde{\mr{CAT}}^\mr{id}(M) \longto B\mr{haut}^\mr{id}(M).\]
By Proposition \ref{prop.haut}(i) its base is in $\cat{Fin}$, so using Lemma \ref{lem.les} it suffices to prove that the fiber has finitely generated homotopy groups. This uses the surgery exact sequence (\ref{eqn:surgery}). We will first prove that the path components of $\mr{Map}(M,G/\mr{CAT})$ and $\bL(M)$ have finitely generated homotopy groups.  Note that $\pi_i(G) \cong \pi_i(\bS)$ for $i \geq 1$, while $\pi_0(G) = \bZ/2\bZ$. In the case $\mr{CAT} = \mr{Diff}$, $\mr{colim}_{n \to \infty} \mr{CAT}(n) \simeq O$ and by Bott periodicity the homotopy groups of $O$ are finitely generated, so that $G/O \in \cat{Fin}$. From this and Theorem \ref{thm.ks}, we can conclude that $G/\mr{CAT} \in \cat{Fin}$ and it has abelian $\pi_1$ because $G/\mr{CAT}$ is an infinite loop space \cite[chapter 6]{madsenmilgram}. By Lemma \ref{lem.sectionspace} each component of $\mr{Map}(M,G/\mr{CAT})$ has finitely generated homotopy groups. For the second term, $\pi_i(\bL(M)) \cong \pi_{i +n}(\bL(\bZ[\pi_1(M)]))$, which is finitely generated when $\pi_1(M)$ is finite \cite{hambletontaylor}. 

That $\mr{haut}^\mr{id}(M)/\widetilde{\mr{CAT}}^\mr{id}(M)$ has finitely generated homotopy groups then follows from  Lemma \ref{lem.les}(iii), (iii'), where we use that condition (a) since $\pi_1(\mr{Map}(M,G/\mr{CAT}))$ is abelian again because $G/\mr{CAT}$ is an infinite loop space. We conclude that $B\smash{\widetilde{\mr{CAT}}^\mr{id}}(M) \in \cat{Fin}$, which finishes the proof of part (i) of this Proposition.

\item We intend to apply Lemma \ref{lem.propfinitetypespaces}(ii) to the fiber sequence
\[B\widetilde{\mr{CAT}}^\mr{id}(M) \longto B\widetilde{\mr{CAT}}(M) \longto B\pi_0(\widetilde{\mr{CAT}}(M)),\] and it thus suffices to prove that $B\pi_0(\widetilde{\mr{CAT}}(M)) \in \cat{HFin}$. First, by Lemma \ref{lem.pseudokerfg} the classifying space of the kernel of the surjection $\pi_0(\mr{CAT}(M)) \to \pi_0(\widetilde{\mr{CAT}}(M))$ is in $\cat{HFin}$. Next, the group $\pi_0(\mr{CAT}(M))$ is arithmetic when $\dim M \geq 5$ and $\pi_1(M)$ is finite \cite{trianta}, and thus by Theorem \ref{thm.arithmetic} we have $B\pi_0(\mr{CAT}(M)) \in \cat{HFin}$. Using part (i) of Lemma \ref{lem.groupsprops}, which says that a group with classifying spaces in $\cat{HFin}$ are closed under quotients, we conclude that $\pi_0(\widetilde{\mr{CAT}}(M)) \in \cat{HFin}$. If $n=4$ and $\pi_1(M) = 0$, then $\pi_0(\mr{Top}(M))$ is arithmetic and pseudoisotopy implies isotopy \cite{quinnisotopy}.\qedhere
\end{enumerate}\end{proof}

\begin{lemma}\label{lem.pseudokerfg} If $n \geq 5$, $\pi_1(M)$ is finite and $M$ is oriented, then the kernel $K$ of the map $\pi_0(\mr{CAT}(M)) \to \pi_0(\widetilde{\mr{CAT}}(M))$ has classifying space in $\cat{HFin}$.\end{lemma}

\begin{proof}\cite[Proposition II.5.1]{hatcherwagoner} describes the kernel $K$ as the quotient of the group $\pi_0(C(M))$ by the abelian subgroup $\pi_0(\mr{Diff}_\partial(M \times I))$, with $C(M)$ as in Definition \ref{def.concordance}. To apply Lemma \ref{lem.groupsprops}(i) it suffices to show $\pi_0(C(M))$ is finitely generated nilpotent, as then $\pi_0(\mr{Diff}_\partial(M \times I))$ will be finitely generated abelian and by Lemma's \ref{lem.emhfin} and \ref{lem.groupsprops}(ii) both $B\pi_0(\mr{Diff}_\partial(M \times I))$ and $B\pi_0(C(M))$ will be in $\cat{HFin}$. \cite[Theorem 3.1]{hatcherconcordance} gives an exact sequence
\begin{align}\label{eqn.hatcher} \begin{split} H_0(\pi_1(M);&\pi_2(M)[\pi_1(M)]/\pi_2(M)[1]) \longto \pi_0(C(M)) \\
& \longto \mr{Wh}_2(\pi_1(M)) \oplus H_0(\pi_1(M);\bZ/2\bZ[\pi_1(M)]/\bZ/2\bZ[1]) \longto 0,\end{split} \end{align}
with $\mr{Wh}_2(\pi_1(M))$ a quotient of $K_2(\bZ[\pi_1(M)])$. Because $\pi_1(M)$ is finite, the universal cover $\tilde{M}$ is a finite CW-complex and thus $\pi_2(M) \cong H_2(\tilde{M})$ is finitely generated. Both $H_0$-terms of (\ref{eqn.hatcher}) are thus coinvariants of actions on finitely generated abelian groups, and hence finitely generated abelian. For $\mr{Wh}_2(\pi_1(M))$ we use that $K_2(\bZ[G])$ is finitely generated abelian if $G$ is finite by \cite[Theorem 1.1.(i)]{kuku}. %Now use Theorem \ref{thm.arithmetic} and Lemma \ref{lem.groupsprops}(i), (ii).
\end{proof}

Finally we deduce a result about the difference between block diffeomorphisms and diffeomorphisms, previously known only in the concordance stable range.

\begin{corollary}If $M$ is closed smooth 2-connected manifold of dimension $6$ or $\geq  8$, then $\widetilde{\mr{CAT}}(M)/\mr{CAT}(M)$ is in $\cat{\Pi Fin}$.\end{corollary}

\begin{proof}Under the conditions of the corollary, pseudoisotopy classes coincide with isotopy classes \cite{cerf,rourkeemb,blr}. Thus the map $\pi_0(\mr{CAT}(M)) \to \pi_0(\widetilde{\mr{CAT}}(M))$ is an isomorphism and hence there is a fiber sequence
\[\widetilde{\mr{CAT}}(M)/\mr{CAT}(M) \longto B\mr{CAT}^\mr{id}(M) \longto B\widetilde{\mr{CAT}}^\mr{id}(M).\]
From Lemma \ref{lem.les}(iii), (iii'), Corollaries \ref{cor.evenmfd}, \ref{cor.oddmfd} and \ref{cor.mfdcat}, and Proposition \ref{prop.block} we conclude that $\widetilde{\mr{CAT}}(M)/\mr{CAT}(M) \in \cat{Fin}$.
\end{proof}

\bibliographystyle{amsalpha}
\bibliography{cell}

\vspace{.5cm}

\end{document}